\documentclass[a4, reqno, 11pt, pdftex]{amsart}

\usepackage{amsmath}
\usepackage{amsthm}
\usepackage{amssymb} 
\usepackage{amsfonts} 
\usepackage{ascmac}
\usepackage[bbgreekl]{mathbbol}
\usepackage{enumerate}
\usepackage[all]{xy}
\usepackage{tikz}
\usepackage{geometry}
\usepackage{lscape}
\usepackage{longtable}
\usepackage{nicematrix}
\renewcommand{\textcolor}[2]{#2}

\makeatletter
\@namedef{subjclassname@2020}{%
  \textup{2020} Mathematics Subject Classification}
\makeatother

\newcommand{\cG}{\mathcal{G}}

\newcommand{\cQ}{\mathcal{Q}}
\newcommand{\bbk}{\mathbb{k}}

\DeclareMathOperator{\Ext}{Ext}
\DeclareMathOperator{\HH}{HH}
\DeclareMathOperator{\Image}{Im}
\DeclareMathOperator{\Ker}{Ker}
\DeclareMathOperator{\rank}{rank}
\DeclareMathOperator{\tr}{tr}
\DeclareMathOperator{\Hom}{Hom}

\DeclareMathOperator{\tails}{\mathsf{tails}}
\DeclareMathOperator{\coh}{\mathsf{coh}}
\DeclareMathOperator{\Db}{\mathsf{D^{b}}}
\DeclareMathOperator{\mods}{\mathsf{mod}}

\DeclareMathSymbol{\Xdsum}{\mathop}{largesymbols}{88}
\DeclareMathSymbol{\Xtsum}{\mathop}{largesymbols}{80}
\DeclareMathOperator*{\dsum}{\mathchoice{\Xdsum}{\Xdsum}{\Xtsum}{\Xtsum}}
\newcommand\Dsum{\dsum\limits}

\theoremstyle{plain} 
\newtheorem{thm}{Theorem}[section]
\newtheorem{cor}[thm]{Corollary}
\newtheorem{lem}[thm]{Lemma}
\newtheorem{prop}[thm]{Proposition}
\newtheorem*{thmnonumber}{Theorem}

\theoremstyle{definition}
\newtheorem{dfn}[thm]{Definition}
\newtheorem{rem}[thm]{Remark}

\numberwithin{equation}{section}

\title[Hochschild cohomology of Beilinson alg. of graded down-up alg.]
{Hochschild cohomology of Beilinson algebras of graded down-up algebras with weights ($n,m$)}

\author{Ayako Itaba}
\address{
Katsushika Division, 
Institute of Arts and Sciences, 
Tokyo University of Science, 
6-3-1 Niijuku, Katsushika-ku, Tokyo, 125-8585, JAPAN}
\email{itaba@rs.tus.ac.jp}

\author{Shu Minaki}
\address{
Department of Mathematics, 
Graduate School of Science, 
Tokyo University of Science, 
1-3 Kagurazaka, Shinjuku-ku, Tokyo, 162-8601, JAPAN}
\email{1125704@ed.tus.ac.jp} 
\begin{document}
\begin{abstract}
Let $A=A(\alpha, \beta)$ be a graded down-up algebra 
with weights $(\mathrm{deg}\, x, \mathrm{deg}\, y)=(n,m)$
and $\beta \neq 0$, and $\nabla A$ its Beilinson algebra.
Such an algebra $A$ is a 3-dimensional cubic AS-regular algebra 
by Kirkman--Musson--Passman.
Assuming $\gcd(n, m)=1$ and $m \geq n$, we extend the previous results 
on the Hochschild cohomology of $\nabla A$. 
Known cases include $(n,m) = (1,1)$ (Belmans) and $(n = 1,\,m \geq 2)$ (Itaba--Ueyama).
In this paper, we determine the dimensions of the Hochschild cohomology groups 
of $\nabla A$ in the remaining case $n\geq 2$ and $m\geq 2$ by explicitly constructing the projective resolution and computing the ranks of the arising representation matrices. 
As a byproduct, for $m>n>1$, we show that the derived category of the noncommutative projective scheme associated to $A$ is not equivalent to the derived category of any smooth projective surface. Moreover, for all $m \geq  n \geq 1$, we describe the ring structure of the Hochschild cohomology group $\nabla A$ with respect to the Yoneda product.
\end{abstract}
\maketitle
\section{Introduction}
Throughout this paper, let $\bbk$ be an algebraically closed field 
\textcolor{red}{of characteristic zero}. 
\textcolor{red}{Artin--Schelter \cite{AS} introduced the notion of an AS-regular algebra over $\bbk$.}
Note that an AS-regular algebra is a noncommutative analogue of 
\textcolor{red}{a commutative polynomial ring and plays a central role in noncommutative algebraic geometry.} 
\textcolor{blue}{Using an algebro-geometric approach}
to noncommutative algebra, 
noncommutative projective schemes associated to AS-regular algebras  
\textcolor{blue}{were formalized by Artin--Zhang \cite{AZ}.}
\textcolor{red}{
These are defined as the Serre quotient of} 
\textcolor{blue}{the category of finitely generated graded right $A$-modules by the subcategory of finite-dimensional right $A$-modules}. 
\textcolor{red}
{
An important result of Minamoto--Mori \cite[Theorem 4.14]{MM} 
is as follows: 
If $A$ is a coherent $d$-dimensional AS-regular algebra 
and $\nabla A$ is its Beilinson algebra, then the global dimension of $\nabla A$ is $d-1$}
\textcolor{red}{and there is a triangulated equivalence $\Db(\tails A) \cong \Db(\mods \nabla A)$. Here $\mods \nabla A$ denotes the category of finitely generated right $\nabla A$-modules 
and $\tails A$ denotes the noncommutative projective scheme associated to $A$;}
\textcolor{blue}{$\Db(\tails A)$ and $\Db(\mods \nabla A)$ are the bounded derived categories} 
\textcolor{red}{of $\tails A$ and $\mods \nabla A$, respectively.}

\textcolor{red}{In this paper, we focus on graded down-up algebras.}
\textcolor{red}{For $\alpha,\,\beta\in \bbk$ and positive integers $n,\,m$, 
{\em the graded down-up algebra with weights $(n,m)$} is the graded 
$\bbk$-algebra 
$$
A(\alpha, \beta) := \bbk\langle{x,y} \rangle/(x^{2}y-\beta yx^{2} -\alpha xyx,\; xy^{2}-\beta y^{2}x -\alpha yxy), 
$$
with $\deg x=n, \deg y=m$.}
\textcolor{red}{This down-up algebra was defined by Benkart--Roby \cite{BR} for studying a poset.}
\textcolor{blue}{Kirkman--Musson--Passman \cite{KMP} showed that a down-up algebra $A(\alpha,\beta)$ is a $3$-dimensional noetherian AS-regular algebra if and only if $\beta \neq 0$.}
\textcolor{red}{In that case, its Gorenstein parameter equals $2(n + m)$.}
Let $A:=A(\alpha, \beta)$ be a down-up algebra 
with weights $\deg\, x=n$, and $\deg\, y =m$, $\beta \neq 0$
\textcolor{red}
{and $\nabla A$ its Beilinson algebra. }
\textcolor{red}{In \cite{IU}, the first author and Ueyama remarked that 
if $n$ and $m$ have a common factor, then the Beilinson algebra $\nabla A$ decomposes into a direct sum (or product) of algebras corresponding to weight-pairs $(n',\,m')$ with $\gcd(n',\,m') = 1$; hence it suffices to study the coprime case.}

\textcolor{red}{By Carvalho--Musson \cite[4.1 Lemma]{CM}, we may assume $m \geq n$} without loss of generality. 
\textcolor{red}{Thus it is enough to consider the case $\gcd(n,m) = 1$ with $m \geq n$ 
when studying the Hochschild cohomology of $\nabla A$.}
\textcolor{red}{For the case $n=m=1$}, Belmans \cite{Bel} 
\textcolor{red}{computed} the dimensions of the Hochschild cohomology groups and the Lie structures of the first Hochschild cohomology by using a geometric approach. 
\textcolor{red}{For the case $m \geq n=1$}, the first author and Ueyama \cite{IU} 
\textcolor{red}{computed} the dimensions of the Hochschild cohomology groups 
by using an algebraic approach. 
\textcolor{blue}{In this paper, we determine the dimensions} of the Hochschild cohomology groups $\HH^{r}(\nabla A)$ \textcolor{red}{of $\nabla A$} \textcolor{red}{for} the case $m>n>1$ 
by the same algebraic approach \textcolor{red}{used} in \cite{IU}. 
One of the main results in this paper is as follows: 
\begin{thmnonumber}[Theorem \ref{HHdim-IM}]\label{main-thm}
Let $A=A(\alpha, \beta)$ be a graded down-up algebra with weights $\mathrm{deg}\ x=n$, $\mathrm{deg}\ y=m$ and $\beta\neq 0$ 
where $m \geq n >1$ and $\mathrm{gcd}(n, m)=1$, and $\nabla A$ \textcolor{red}{its} Beilinson algebra.
Then we obtain the \textcolor{red}{dimension} formula of the Hochschild cohomology groups 
$\HH^{r}(\nabla A)$ 
as follows:
\begin{itemize}
\item $\dim_{\bbk} \HH^0(\nabla A)=1;$
\item 
$\text{\textcolor{red}{$\dim_{\bbk}$}} \HH^1(\nabla A)=
\begin{cases}
2 \quad \text{\textup{if $n+m$ is even and $\alpha=0$}},\\
1 \quad \text{\textup{if $n+m$ is odd or $\alpha\neq 0$}};
\end{cases}$
\item $\dim_{\bbk} \HH^{2}(\nabla A)=
\begin{cases}
m+5 \quad \text{\textup{if $n=2$, $n+m$ is even, and $\alpha=0$}},\\
m+4 \quad \text{\textup{if $n=2$ and $n+m$ is odd or $\alpha\neq 0$}},\\
n+m+1 \quad  \text{\textup{if $n \geq 3$, $n+m$ is even, and $\alpha=0$}},\\
n+m \quad \text{\textup{if $n \geq 3$ and $n+m$ is odd or $\alpha\neq 0$}};
\end{cases}$
\item $\dim_{\bbk} \HH^{r}(\nabla A)= 0$ for $r\geq 3$. 
\end{itemize}
\end{thmnonumber}
\textcolor{red}{As a byproduct, we obtain structural consequences for the derived categories associated to these algebras. 
In particular, for $m>n>1$, we show that $\Db(\tails A)$ is not equivalent to the bounded derived category of any smooth projective surface (Remark \ref{rem-cor}). 
In addition, we compute explicit bases for $\mathrm{HH}^{1}(\nabla A)$ 
and $\mathrm{HH}^{2}(\nabla A)$ (Propositions \ref{basHH-1} and \ref{basHH-2}) and determine the ring structure of $\mathrm{HH}^{\ast}(\nabla A)$ with respect to the Yoneda product in the case 
$m \geq n \geq 1$ (Theorem \ref{thm-ring}).}

\textcolor{red}{The paper is organized as follows: 
Section \ref{Prelimi} reviews definitions and known facts about AS-regular algebras, Beilinson algebras, graded down-up algebras, and Hochschild cohomology, and recalls prior dimensional results. 
Section \ref{HHgroups} carries out the rank computations and proves the main dimensional theorem (Theorem \ref{HHdim-IM}). 
Section \ref{sec-basis} constructs bases of $\mathrm{HH}^{1}(\nabla A)$ 
and $\mathrm{HH}^{2}(\nabla A)$, 
and Section \ref{sec-5} describes the Yoneda ring structure on $\mathrm{HH}^{\ast}(\nabla A)$.}

\section{Preliminaries}
\label{Prelimi}
In this section, we recall the definition of a down-up algebra and 
present our motivation for computing the Hochschild cohomology group of the Beilinson algebra of a down-up algebra. 

First, we recall the definition of \textcolor{red}{an AS-regular algebra}. 
This algebra has an important role in noncommutative algebraic geometry. 
\begin{dfn}[{\cite{AS}}]
Let $A=\bigoplus_{i \geq 0} A_{i}$ be a finitely graded algebra with \textcolor{magenta}{$A_{0}=\bbk$}. 
If $A$ satisfies the following conditions, then $A$ is called 
\textcolor{red}{a $d$-dimensional} \textit{AS-regular algebra} 
of the Gorenstein parameter of $\ell$\textup{:}
\begin{enumerate}[(i)]
\item $\mathrm{gl.dim} A =d <\infty$; 
\item
$
\textcolor{magenta}{
  \underline{\Ext}^{i}_{A}(\bbk, A)
  }
  \cong
\begin{cases}
{\bbk}(\ell) & \text{if}\, i=d,\\
0 & \text{if}\, i \neq d;\\
\end{cases}
$
\quad (\textit{Gorenstein condition}); 
\item $\mathrm{GKdim} A
:=\mathrm{inf}\, 
	\left\{
	\alpha \in \textcolor{magenta}{\mathbb{R}_{\geq 0}} \mid \mathrm{dim}_{\bbk} \left(\sum_{i=0}^{n} A_i\right) \leq n^{\alpha} \;\,\textnormal{for all}\;\, n \gg 0
	\right\}
< \infty $

\noindent
 ($\mathrm{GKdim} A$ is called the \textit{Gelfand--Kirillov dimension} of $A$). 
\end{enumerate}
\end{dfn}
\begin{dfn}[see {\cite{MM}}]
Let $A=\bigoplus_{i \geq 0} A_{i}$ be 
\textcolor{red}{a $d$-dimensional} AS-regular algebra 
and of the Gorenstein parameter $\ell$. 
The \textit{Beilinson algebra} $\nabla A$ of $A$ is defined by
\begin{align*}
\nabla A:=
\left(\begin{array}{ccccc}
A_{0} & A_{1} & \cdots & A_{\ell-2} & A_{\ell-1}\\
0 & A_{0}  & \cdots & A_{\ell-3} & A_{\ell-2}\\ 
\vdots && \ddots && \vdots \\
\vdots &&& \ddots  & \vdots \\
\vspace{-2pt}\\
0 & \cdots &  \cdots & 0 & A_{0}\\ 
\end{array}\right)
\quad \text{with the multiplication $(a_{ij})(b_{ij})=\left( \sum _{r=0}^{\ell-1}a_{rj}b_{ir} \right)$.}
\end{align*}
\end{dfn}
Let $A$ be an AS-regular algebra 
\textcolor{red}{
and $\mods \nabla A$ denote a category of finitely generated right $\nabla A$-modules.}
A category $\tails A$ is defined by the Serre quotient of the category of graded right $A$-modules by the subcategory of finite\textcolor{red}{-}dimensional right $A$\textcolor{red}{-}modules. 
This category $\tails A$ is called the \textit{noncommutative projective scheme associated to $A$}, 
\textcolor{blue}{which is formalized by Artin--Zhang \cite{AZ}.}

\textcolor{red}{
We recall the following result of Minamoto--Mori \cite[Theorem 4.14]{MM} relating a noetherian AS-regular algebra $A$ and its Beilinson algebra $\nabla A$.}
\begin{thm}\label{thm-MM}
Let $A$ be a coherent \textcolor{red}{$d$-dimensional} AS-regular algebra 
and $\nabla A$ \textcolor{red}{its} Beilinson algebra. 
Then the global dimension of $\nabla A$ is $d-1$ 
\textcolor{red}{and} there exists a triangulated category equivalence 
$
\Db(\tails A) \cong \Db(\mods \nabla A)
$. 
\end{thm}


\textcolor{blue}{Theorem \ref{du->AS} motivated us to study a graded down-up algebra in this paper.}
\begin{thm}[{\cite{KMP}}]\label{du->AS}
Let $A(\alpha, \beta)$ be a graded down-up algebra with weights  
$\deg x = n$, $\deg y = m$ and $\beta \neq 0$. 
Then $A(\alpha, \beta)$ is \textcolor{red}{a $3$-dimensional} AS-regular algebra 
and the Gorenstein parameter $\ell =2(n+m)$. 
\end{thm}

\textcolor{red}{We recall the definition of the Hochschild cohomology of a $\bbk$-algebra $B$. 
The {\it $r$-th Hochschild cohomology group}}
is defined by
\[
\HH^{r}(B):=\Ext^{r}_{B^{\mathrm e}}(B, B)
\quad
(r \geq 0)
\]
where $B^{\mathrm e}\textcolor{red}{:}=B^{\mathrm op}\otimes_{\bbk} B$ is the enveloping algebra of $B$. 
\begin{prop}[{\cite{IU}}]\label{wolg-1}
Let $A:=A(\alpha, \beta)$ be a graded down-up algebra with weights $\deg x= n$ and $\deg y =m$.
We 
\textcolor{red}{set}  
$A':=A'(\alpha, \beta)$ as a graded down-up algebra with the generators 
$x', y'$ and weights $\deg x= n/k$ and $\deg y =m/k$ where $k:=\gcd(n, m)$.
Assume that $\beta \neq 0$. 
Then the Hochschild cohomology groups of the Beilinson algebras of the above down-up algebras 
satisfy
$
\HH^{r}(\nabla A)\cong (\HH^{r}(\nabla A'))^{k}\,
\text{for all $r\geq 0$}
$. 
\end{prop}

By Proposition \ref{wolg-1}, we can assume that $\gcd(\deg x,\deg y)=\gcd(n, m)=1$ 
without loss of generality. 
\begin{lem}[{\cite[4.1. Lemma (i)]{CM}}]
Let $A(\alpha, \beta)$ be a graded down-up algebra with weights $(n,\,m)$. 
If $\beta \neq 0$, then $A(\alpha, \beta) \cong A(-\alpha\beta^{-1}, \beta^{-1})$ via the map interchanging $x$ and $y$. 
\end{lem}
Hence, we can assume that $\deg x =n \leq \deg y =m$ without loss of generality.

\textcolor{red}{We} recall the definition of the Yoneda product of Hochschild cohomology groups of algebras (\textcolor{red}{see, for example, }\cite{W}).

\begin{dfn}
Let $B$ be a ${\bbk}$-algebra and $P_{\bullet}=(P_{r}, \phi_{r})$ be a projective resolution of $B$ as right $B^{\mathrm{e}}$-module\textcolor{red}{s}. 
For $[f] \in \HH^{p}(B)$ and $[g] \in \HH^{q}(B)$ represented by $f \in \Hom_{\bbk}(P_{p}, B)$ and $g \in \Hom_{\bbk}(P_{q}, B) $, the \textit{Yoneda product} $[f] \smile [g]$ is the residue class of the ${\bbk}$-linear homomorphism $f\tau_{q}$ defined by as follows\textcolor{red}{: } 
There exists a chain map $(\tau_{\bullet})$ from
\[
\cdots \rightarrow P_{{q}+{p}} \xrightarrow{\phi_{{q}+{p}}} \cdots \xrightarrow{\phi_{{q}+2}} P_{{q}+1} \xrightarrow{\phi_{{q}+1}} P_{q}
\xrightarrow{g} B \rightarrow 0 \rightarrow \cdots
\]
to
\[
\cdots \rightarrow P_{p} \xrightarrow{\phi_{p}} \cdots \xrightarrow{\phi_{2}} P_{1} \xrightarrow{\phi_{1}} P_{0}
\xrightarrow{\phi_{0}} B \rightarrow 0 \rightarrow \cdots
\]
such that $\tau_{-1}=id_{B}$, 
that is, 

\hspace{50pt}
\xymatrix{
\cdots \ar[r] & P_{{q}+{p}} \ar[r]^{\phi_{{q}+{p}}}  \ar[d]^{\tau_{p}} & \cdots \ar[r]^{\phi_{{q}+2}} & P_{{q}+1} \ar[r]^{\phi_{{q}+1}} \ar[d]^{\tau_{1}} & P_{q} \ar[r]^{g} \ar[d]^{\tau_{0}} & B \ar[r] \ar[d]^{\tau_{-1}=\mathrm{id}_{B}} & 0 \ar[r] & \cdots \\
\cdots \ar[r] & P_{p} \ar[r]^{(-1)^{q}\phi_{p}} \ar[d]^{f} & \cdots \ar[r]^{(-1)^{q}\phi_{2}} & P_{1} \ar[r]^{(-1)^{q}\phi_{1}} & P_{0} \ar[r]^{\phi_{0}} & B \ar[r] & 0 \ar[r] & \cdots \\
& B
}
\end{dfn}

\begin{thm}[{see \cite[Section 2.2]{W}}]\label{gr-com}
The Yoneda product on the Hochschild cohomology group $\HH^{r}(B)$ of a ${\bbk}$-algebra $B$ is graded commutative, that is, $\mathsf{f} \in \HH^{p}(B)$ and $\mathsf{g} \in \HH^{q}(B)$ satisfy
$
\mathsf{f} \smile \mathsf{g}=(-1)^{{p}{q}}\mathsf{g} \smile \mathsf{f}
$. 
\end{thm}

\textcolor{red}{Throughout this paper, we assume $\gcd (n,\,m)=1$ without loss of generality. 
Moreover, by \cite[Theorem 1.4]{IU}, it suffices to consider the case $\gcd (n,\,m)=1$.}
\textcolor{red}{When} $\deg x =n =1$, the dimensions of the Hochschild cohomology group of the Beilinson algebra of the graded down-up algebra were already 
\textcolor{red}{given} by Belmans \cite{Bel} \textcolor{red}{(for $n=m=1$)}
and by the first author and Ueyama \cite{IU} \textcolor{red}{(for $n=1$ and $m\geq 2$)}. 

\textcolor{red}{Before recalling the previous results of \cite{Bel} and \cite{IU}, 
we recall some notation used below. Define the sequence}
\begin{align*}
\lambda_{r+2}\textcolor{red}{:}=\alpha \lambda_{r+1} + \beta \lambda_{r}
\text{ with }
\lambda_{0}\textcolor{red}{:}=0 \, \text{ and } \, \lambda_{-1} \textcolor{red}
{:}= \beta^{-1}.
\end{align*}
Here, we \textcolor{red}{introduce} the following conditions:

%
\begin{center}
\renewcommand{\arraystretch}{1.3}
\begin{tabular}{|c|c|p{60pt}|c|c|}
%
\multicolumn{2}{c}{Condition $1$}& \multicolumn{1}{c}{} &\multicolumn{2}{c}{Condition $2$}
\\ \cline{1-2} \cline{4-5}
%
\textrm{Case I} & $n+m$ is even and $\alpha=0$
&
&
Case $\mathrm{1}$ & $\lambda_{m+1}=0$
\\ \cline{1-2} \cline{4-5}
%
\textrm{Case I\hspace{-1.2pt}I} & otherwise
&
&
Case $\mathrm{2}$ & $\lambda_{m+1} \neq 0$ and $\alpha^{2}+4\beta=0$
\\ \cline{1-2} \cline{4-5}
%
\multicolumn{2}{c}{}
&
&
Case $\mathrm{3}$ & $\lambda_{m+1} \neq 0$ and $\alpha^{2}+4\beta \neq0$
\\ \cline{4-5}
\end{tabular}
\renewcommand{\arraystretch}{1.0}
\end{center}
\begin{thm}
[{\cite[Table 2]{Bel}}]\label{thm-Bel}
Let $A=A(\alpha, \beta)$ be a graded down-up algebra with weights 
$\mathrm{deg}\ x =n =1$, $\mathrm{deg}\ y =m =1$, $\beta \neq 0$, 
and $\nabla A$ \textcolor{red}{its} Beilinson algebra.
Then we obtain the \textcolor{red}{dimension} formula of the Hochschild cohomology groups 
$\HH^{r}(\nabla A)$ 
as follows\textup{:}
\begin{itemize}
\item $\dim_{\bbk} \HH^0(\nabla A)=1;$
\item 
$\dim_{\bbk} \HH^1(\nabla A)=
\begin{cases}
6 \quad \text{\textup{for Case I (and Case $1$)}} , \\
3 \quad \text{\textup{for Case I\hspace{-1.2pt}I and Case $2$}} ,\\
1 \quad \text{\textup{for Case I\hspace{-1.2pt}I and Case $3$}} ;
\end{cases}$
\item $\dim_{\bbk} \HH^2(\nabla A)=
\begin{cases}
9 \quad \text{\textup{for Case I (and Case $1$)}}, \\
6 \quad \text{\textup{for Case I\hspace{-1.2pt}I and Case $2$}},\\
4 \quad \text{\textup{for Case I\hspace{-1.2pt}I and Case $3$}};
\end{cases}$
\item $\dim_{\bbk} \HH^{r}(\nabla A)= 0$ for ${r}\geq 3$. 
\end{itemize}
\end{thm}
\begin{thm}
[{\cite[Theorem1.4]{IU}}]\label{thm-IU}
Let $A=A(\alpha, \beta)$ be a graded down-up algebra with weights 
$\mathrm{deg}\,x =n =1$, $\mathrm{deg}\,y =m > 1$, $\beta \neq0$, 
and $\nabla A$ \textcolor{red}{its} Beilinson algebra.
Then we obtain the \textcolor{red}{dimension} formula of the Hochschild cohomology groups 
$\HH^{r}(\nabla A)$ 
as follows\textup{:}
\begin{itemize}
\item $\dim_{\bbk} \HH^0(\nabla A)=1;$
\item 
$\dim_{\bbk} \HH^1(\nabla A)=
\begin{cases}
4 \quad \text{\textup{for Case I (and Case $1$)}}, \\
3 \quad \text{\textup{for Case I\hspace{-1.2pt}I and Case $1$}},\\
2 \quad \text{\textup{for Case I\hspace{-1.2pt}I and Case $2$}},\\
1 \quad \text{\textup{for Case I\hspace{-1.2pt}I and Case $3$}};
\end{cases}$
\item $\dim_{\bbk} \HH^2(\nabla A)=
\begin{cases}
8 \quad \text{\textup{for $m=2$, Case I\hspace{-1.2pt}I, and Case $1$}}, \\
7 \quad \text{\textup{for $m=2$, Case I\hspace{-1.2pt}I, and Case $2$}}, \\
6 \quad \text{\textup{for $m=2$, Case I\hspace{-1.2pt}I, and Case $3$}}, \\
m+5 \quad \text{\textup{for $m \geq 3$ and Case I (and Case $1$)}}, \\
m+4 \quad \text{\textup{for $m \geq 3$, Case I\hspace{-1.2pt}I, and Case $1$}},\\
m+3 \quad \text{\textup{for $m \geq 3$, Case I\hspace{-1.2pt}I, and Case $2$}},\\
m+2 \quad \text{\textup{for $m \geq 3$, Case I\hspace{-1.2pt}I, and Case $3$}};
\end{cases}$
\item $\dim_{\bbk} \HH^{r}(\nabla A)= 0$ for ${r}\geq 3$. 
\end{itemize}
\end{thm}
\section{Hochschild cohomology groups}
\label{HHgroups}
\textcolor{red}{In this section, using the algebraic methods of \cite{IU}, 
we compute the Hochschild cohomology groups of the Beilinson algebra $\nabla A$ of the graded down-up algebra $A=A(\alpha, \beta)$ (with $\mathrm{deg}\, x=n$, $\mathrm{deg}\, y=m$ and $\beta\neq0$), under the assumptions $m \geq n >1$ and $\mathrm{gcd}(n, m)=1$.}

\textcolor{red}{Since the Beilinson algebra $\nabla A$ is finite-dimensional over ${\bbk}$
by Theorem \ref{thm-MM}, it can be presented by the quiver 
$${\cQ}_{\nabla A}=(({\cQ}_{\nabla A})_{0}, ({\cQ}_{\nabla A})_{1}, s, t)$$ 
with an admissible ideal $\mathcal{I}$ of ${\bbk}{\cQ}_{\nabla A}$. 
The vertex and arrow sets are given by}
%
\[
\begin{cases}
({\cQ}_{\nabla A})_{0} := \{ e_{1},\, e_{2},\, \ldots, e_{2(n+m)} \} ,\\
({\cQ}_{\nabla A})_{1} := \{ x_{1},\, x_{2}, \ldots,\,x_{n+2m},\, y_{1},\, y_{2}, \ldots,\,y_{2n+m} \} ,\\
s(x_{i}):=e_{i}, t(x_{i})=e_{i+n} \ \text{for} \ 1 \leq i \leq n+2m, \\
s(y_{j}):=e_{j}, t(y_{j})=e_{j+m} \ \text{for} \ 1 \leq j \leq 2n+m\textcolor{red}{,} \\
\end{cases}
\]
\textcolor{red}{and $\mathcal{I}$ is generated} by the relations
\[
\begin{cases}
f_{i}:= x_{i}x_{i+n}y_{i+2n} -\alpha x_{i}y_{i+n}x_{i+n+m} -\beta y_{i}x_{i+m}x_{i+n+m}
\quad\text{for} \ 1 \leq i \leq m,\\
g_{j}:= x_{j}y_{j+n}y_{j+n+m} -\alpha y_{j}x_{j+m}y_{j+n+m} -\beta y_{j}y_{j+m}x_{j+2m}
\quad\text{for} \ 1 \leq j \leq n.
\end{cases}
\]
For the case $m>n>1$, the simplest case we deal in this section is the case that $(n,m)=(2,3)$. 
\textcolor{red}{The following figure shows the quiver of the Beilinson algebra for $(n, m) = (2, 3)$:}
\[
\hspace{-33pt}
   \xymatrix{
    & 1 \ar@/^{15pt}/@{->}[rr]^-{x_{1}} \ar@/^{-15pt}/@{->}[rrr]_{y_{1}}
    & 2 \ar@/^{15pt}/@{->}[rr]^-{x_{2}} \ar@/^{-15pt}/@{->}[rrr]_{y_{2}}
    & 3 \ar@/^{15pt}/@{->}[rr]^-{x_{3}} \ar@/^{-15pt}/@{->}[rrr]_{y_{3}}
    & 4 \ar@/^{15pt}/@{->}[rr]^-{x_{4}} \ar@/^{-15pt}/@{->}[rrr]_{y_{4}}
    & 5 \ar@/^{15pt}/@{->}[rr]^-{x_{5}} \ar@/^{-15pt}/@{->}[rrr]_{y_{5}}
    & 6 \ar@/^{15pt}/@{->}[rr]^-{x_{6}} \ar@/^{-15pt}/@{->}[rrr]_{y_{6}}
    & 7 \ar@/^{15pt}/@{->}[rr]^-{x_{7}} \ar@/^{-15pt}/@{->}[rrr]_{y_{7}}
    & 8 \ar@/^{15pt}/@{->}[rr]^-{x_{8}} 
    & 9
    & 10
     }.
\]

Note that $\HH^{0}(\nabla A)={\bbk}$ (\textcolor{red}{see, for example,} \cite{LWZ}, \cite{IU}). 
We compute \textcolor{red}{$\HH^{r}(\nabla A)$} 
by the method \textcolor{red}{of} Green--Snashall \cite{GS}. 
Let $(\nabla A)^{\mathrm{e}} = (\nabla A)^{\mathrm{op}} \otimes_{\bbk} \nabla A$ 
\textcolor{red}{denote} the enveloping algebra of $\nabla A$. 
To construct the minimal projective resolution of $\nabla A$ 
as right $(\nabla A)^{\mathrm{e}}$-module\textcolor{red}{s}, 
we define \textcolor{red}{certain} sets and morphisms. 
First, we \textcolor{red}{set} 
\begin{align*}
&\cG^{0}:=({\cQ}_{\nabla A})_{0}, \,
\cG^{1}:=({\cQ}_{\nabla A})_{1}, \,\,
\cG^{2}:= \{ f_{1},\,f_{2},\, \ldots,\, f_{m},\, g_{1},\, g_{2}, \ldots,\, g_{n} \}, \\
&P^{r} := \bigoplus_{h \in \cG^{r}} \nabla A s(h) \otimes_{\bbk} t(h) \nabla A
\quad (r=0, 1, 2). 
\end{align*}
The map $\partial^{0}: P^{0} \rightarrow \nabla A$ is defined by the multiplication map, 
and the right $(\nabla A)^{\mathrm{e}}$-\textcolor{red}{homomorphisms}
$\partial^{r+1}: P^{r+1} \rightarrow P^{r}$ \textcolor{red}{are} defined as follows $(r=0,\,1)$:

\textcolor{red}{For $1 \leq i \leq n+2m,\, 1 \leq j \leq 2n+m$ and $h\in \cG^{1}$, define}
\begin{align*}
\partial^{1}(s(h) \otimes t(h)):=
\begin{cases}
(s(e_{i}) \otimes t(e_{i}))x_{i} - x_{i}(s(e_{i+n}) \otimes t(e_{i+n})) 
& \text{if $h=x_{i}$,}\\
(s(e_{j}) \otimes t(e_{j}))y_{j} - y_{j}(s(e_{j+m}) \otimes t(e_{j+m})) 
& \text{if $h=y_{j}$.}\\
\end{cases}
\end{align*}

\textcolor{red}{For $1 \leq i \leq m,\,1 \leq j \leq n$ and $h\in \cG^{2}$, define}
\begin{align*}
\partial^{2}(s(h) \otimes t(h)):=
\begin{cases}
&(
(s(x_{i}) \otimes t(x_{i}))x_{i+n}y_{i+2n}
+x_{i}(s(x_{i+n}) \otimes t(x_{i+n}))y_{i+2n}\\
&\quad\quad
+x_{i}x_{i+n}(s(y_{i+2n}) \otimes t(y_{i+2n}))
)\\
&\quad
-\alpha(
(s(x_{i}) \otimes t(x_{i}))y_{i+n}x_{i+n+m}
+x_{i}(s(y_{i+n}) \otimes t(y_{i+n}))x_{i+n+m}\\
&\quad\quad+x_{i}y_{i+n}(s(x_{i+n+m}) \otimes t(x_{i+n+m}))
)\\
&\quad
-\beta(
(s(y_{i}) \otimes t(y_{i}))x_{i+m}x_{i+n+m}
+y_{i}(s(x_{i+m}) \otimes t(x_{i+m}))x_{i+n+m}\\
&\quad\quad
+y_{i}x_{i+m}(s(x_{i+n+m}) \otimes t(x_{i+n+m}))
)\quad\quad\text{if $h=f_{i}$,}\\
&(
(s(x_{j}) \otimes t(x_{j}))y_{j+n}y_{j+n+m}
+x_{j}(s(y_{j+n}) \otimes t(y_{j+n}))y_{j+n+m}\\
&\quad\quad+x_{j}y_{j+n}(s(y_{j+n+m}) \otimes t(y_{j+n+m}))
)\\
&\quad
-\alpha(
(s(y_{j}) \otimes t(y_{j}))x_{j+m}y_{j+n+m}
+y_{j}(s(x_{j+m}) \otimes t(x_{j+m}))y_{j+n+m}\\
&\quad\quad
+y_{j}x_{j+m}(s(y_{j+n+m}) \otimes t(y_{j+n+m}))
)\\
&\quad
-\beta(
(s(y_{j}) \otimes t(y_{j}))y_{j+m}x_{j+2m}
+y_{j}(s(y_{j+m}) \otimes t(y_{j+m}))x_{j+2m}\\
&\quad\quad
+y_{j}y_{j+m}(s(x_{j+2m}) \otimes t(x_{j+2m}))
)\quad\quad\text{if $h=g_{j}$.}\\
\end{cases}
\end{align*}
\begin{lem}
The sequence
\begin{align}\label{resol-1}
P^{\bullet}: 0 \rightarrow P^{2}  \xrightarrow{\partial^{2}} P^{1}  \xrightarrow{\partial^{1}} P^{0} \xrightarrow{\partial^{0}} \nabla A \rightarrow 0
\end{align}
is the minimal projective resolution of $\nabla A$ as right $(\nabla A)^{\mathrm{e}}$-modules. 
\end{lem}
\begin{proof}
\textcolor{red}{By Happel \cite[Lemma 1.5]{H}, $\mathrm{pd\,}_{(\nabla A)^{\mathrm{e}}}\nabla A=\mathrm{gl.dim}\, \nabla A =2$. 
Hence the assertion follows from \cite[Theorem 2.9]{GS}.}
\end{proof}
\textcolor{red}{Applying} the functor 
$\widehat{-} := \Hom_{(\nabla A)^{\mathrm{e}}}(-, \nabla A)$ to \eqref{resol-1} 
\textcolor{red}{yields} the complex
\begin{align}\label{resol-2}
0 \rightarrow \widehat{P^{0}} \xrightarrow{\widehat{\partial^{1}}} \widehat{P^{1}} \xrightarrow{\widehat{\partial^{2}}} \widehat{P^{2}} \rightarrow 0.
\end{align}
\subsection{Constructions of the representation \textcolor{blue}{matrices} of $\widehat{\partial^{2}}$}
\label{subsec-3-1}
\textcolor{red}{To} compute the dimensions of the Hochschild cohomology groups 
\textcolor{red}{$\HH^{r}(\nabla A)$}, 
we first give a ${\bbk}$-basis of $\widehat{P^{i}}\,(\forall i \geq 0)$. 
For $e_{i} \in \cG^{0}$, 
\textcolor{red}{define} the right $(\nabla A)^{\mathrm{e}}$-homomorphism 
$\tau_{e_{i}}: P^{0} \rightarrow \nabla A$ by
$
\tau_{e_{i}}(s(h) \otimes t(h)) := 
\begin{cases}
e_{i} &\text{if $h=e_{i}$}, \\
0 &\text{otherwise}
\end{cases}
\quad \text{for $h \in \cG^{0}.$}
$
\begin{flalign*}
\quad
\text{For $x_{i}, y_{j} \in \cG^{1}$, \textcolor{red}{define}}\,
\tau_{x_{i}}(s(h) \otimes t(h)) := 
\begin{cases}
x_{i} &\text{if $h=x_{i}$}, \\
0 &\text{otherwise},
\end{cases}
\tau_{y_{j}}(s(h) \otimes t(h)) := 
\begin{cases}
y_{j} &\text{if $h=y_{j}$}, \\
0 &\text{otherwise}.
\end{cases}&&
\end{flalign*}

\textcolor{red}{If} $n=1$ and $m \geq 1$, \textcolor{red}{then define, }
$
\tau_{y_{j}}^{x^{m}}(s(h) \otimes t(h)) := 
\begin{cases}
x_{j} \dots x_{j+m-1} &\text{if $h=y_{j}$}, \\
0 &\text{otherwise}.
\end{cases}
$

\textcolor{red}{If} $n=1$ and $m=1$, \textcolor{red}{then define, } 
for $h \in \cG^{1}$, 
$
\tau_{x_{i}}^{y}(s(h) \otimes t(h)) := 
\begin{cases}
y_{i} &\text{if $h=x_{i}$}, \\
0 &\text{otherwise.}
\end{cases}
$

For $f_{i}, g_{j} \in \cG^{2}$, \textcolor{red}{define} 
\begin{align*}
&\tau_{f_{i}}^{xyx}(s(h) \otimes t(h)) := 
\begin{cases}
x_{i}y_{i+n}x_{i+n+m} &\text{if $h=f_{i}$}, \\
0 &\text{otherwise},
\end{cases}
\tau_{f_{i}}^{yxx}(s(h) \otimes t(h)) := 
\begin{cases}
y_{i}x_{i+m}x_{i+n+m} &\text{if $h=f_{i}$}, \\
0 &\text{otherwise},
\end{cases}
\\
&
\tau_{g_{j}}^{yxy}(s(h) \otimes t(h)) := 
\begin{cases}
y_{j}x_{j+m}y_{j+n+m} &\text{if $h=g_{j}$}, \\
0 &\text{otherwise},
\end{cases}
\tau_{g_{j}}^{yyx}(s(h) \otimes t(h)) := 
\begin{cases}
y_{j}y_{j+m}x_{j+2m} &\text{if $h=g_{j}$}, \\
0 &\text{otherwise}.
\end{cases}
\end{align*}

\textcolor{red}{If} $n=1$ and $m \geq 1$, \textcolor{red}{then define}
\begin{align*}
&
\tau_{g_{j}}^{yx^{m+1}}(s(h) \otimes t(h)) := 
\begin{cases}
y_{1}x_{m+1} \dots x_{2m+1} &\text{if $h=g_{j}$}, \\
0 &\text{otherwise},
\end{cases}
\tau_{g_{j}}^{xyx^{m}}(s(h) \otimes t(h)) := 
\begin{cases}
x_{1}y_{2}x_{m+2} \dots x_{2m+1} \\
\quad\quad\quad\text{if $h=g_{j}$}, \\
0 \quad\quad\quad\text{otherwise},
\end{cases}
\\
&
\tau_{f_{i}}^{x^{m+2}}(s(h) \otimes t(h)) := 
\begin{cases}
x_{i} \dots x_{i+m+1}&\text{if $h=f_{i}$}, \\
0 &\text{otherwise},
\end{cases}
\tau_{g_{j}}^{x^{2m+1}}(s(h) \otimes t(h)) := 
\begin{cases}
x_{1} \dots x_{2m+1} &\text{if $h=g_{j}$}, \\
0 &\text{otherwise}.
\end{cases}
\end{align*}

\textcolor{red}{If} $n=1$ and $m=2$, \textcolor{red}{then define} 
$
\tau_{f_{i}}^{y^{2}}(s(h) \otimes t(h)) := 
\begin{cases}
y_{i}y_{i+2} &\text{if $h=f_{i}$}, \\
0 &\text{otherwise}.
\end{cases}
$

\textcolor{red}{If} $n=1$ and $m=1$, \textcolor{red}{then define, }
for $h \in \cG^{2}$, 
\begin{align*}
&
\tau_{g_{j}}^{yx^{2}}(s(h) \otimes t(h)) := 
\begin{cases}
y_{1}x_{2}x_{3} &\text{if $h=g_{j}$}, \\
0 &\text{otherwise},
\end{cases}
\tau_{g_{j}}^{xyx}(s(h) \otimes t(h)) := 
\begin{cases}
x_{1}y_{2}x_{3} &\text{if $h=g_{j}$}, \\
0 &\text{otherwise},
\end{cases}
\\
&
\tau_{f_{i}}^{y^{3}}(s(h) \otimes t(h)) := 
\begin{cases}
y_{1}y_{2}y_{3} &\text{if $h=f_{i}$}, \\
0 &\text{otherwise},
\end{cases}
\tau_{g_{j}}^{y^{3}}(s(h) \otimes t(h)) := 
\begin{cases}
y_{1}y_{2}y_{3} &\text{if $h=g_{j}$}, \\
0 &\text{otherwise.}
\end{cases}
\end{align*}


\begin{lem}[{\cite[Lemma 2.2]{IU}}]\label{bas-IU}
\begin{enumerate}
\item
If $n=1$ and $m=2$, then $\widehat{P^{2}}$ has a ${\bbk}$-basis 

\noindent
$\{
\tau_{f_{i_{1}}}^{yx^2}, \tau_{f_{i_{2}}}^{xyx}, 
\tau_{f_{i_3}}^{x^4}, \tau_{f_{i_4}}^{y^2}, 
\tau_{g_1}^{y^2x}, 
\tau_{g_1}^{yxy}, \tau_{g_1}^{yx^3}, 
\tau_{g_1}^{xyx^2}, \tau_{g_1}^{x^5}
\mid 1\leq i_{1}, i_{2}, i_{3}, i_{4} \leq 2 \}$, 
\textcolor{red}{hence} $\dim_{\bbk} \widehat{P^2}=13$.
\item 
If $n=1$ and $m\geq 3$, then $\widehat{P^{2}}$ has a {${\bbk}$}-basis

\noindent
$\{
\tau_{f_{i_1}}^{yx^2}, \tau_{f_{i_2}}^{xyx}, \tau_{f_{i_3}}^{x^{m+2}}, 
\tau_{g_1}^{y^2x}, \tau_{g_1}^{yxy}, \tau_{g_1}^{yx^{m+1}}, \tau_{g_1}^{xyx^m}, \tau_{g_1}^{x^{2m+1}}
\mid 1 \leq i_{1}, i_{2}, i_{3} \leq m\},$
\textcolor{red}{hence} $\dim_{\bbk} \widehat{P^2}=3m+5$.
\end{enumerate}
\end{lem}

\begin{lem}\label{bas-IM}
\begin{enumerate}
\item $\widehat{P^{0}}$ has a ${\bbk}$-basis $\{\tau_{e_i} \mid 1
\leq i \leq 2(n+m) \}$,
\textcolor{red}{hence} $\dim_{\bbk} \widehat{P^0}=2(n+m)$.
\item If $m = n = 1$, then $\widehat{P^{1}}$ has a ${\bbk}$-basis 
$\{\tau_{x_{i_1}},\, \tau_{x_{i_2}}^{y},\, \tau_{y_{j_1}},\, \tau_{y_{j_2}}^{x} \mid 1\leq i_{1}, i_{2}, j_{1}\,, j_{2} \leq 3 \}$,
\textcolor{red}{hence} $\dim_{\bbk} \widehat{P^1}=12$.
\item If $m>n>1$, then $\widehat{P^{1}}$ has a ${\bbk}$-basis 
$\{\tau_{x_i},\, \tau_{y_j} \mid 1\leq i \leq n+2m, 1\leq j \leq 2n+m \}$, 
\textcolor{red}{hence} $\dim_{\bbk} \widehat{P^1}=3(n+m)$. 
\item 
If $m=n=1$, then $\widehat{P^{2}}$ has a ${\bbk}$-basis 
$\{
\tau_{f_1}^{yx^2},\, \tau_{g_1}^{y^2x},\, \tau_{g_1}^{yxy},\, \tau_{f_1}^{xyx},\, 
\tau_{f_1}^{x^3},\, \tau_{g_1}^{yx^2},\, \tau_{g_1}^{xyx},\, \tau_{g_1}^{x^3},\, 
\tau_{g_1}^{y^3},\, \tau_{f_1}^{y^2x},\, \tau_{f_1}^{yxy},\, \tau_{f_1}^{y^3}
\}$, 

\noindent
\textcolor{red}{hence} $\dim_{\bbk} \widehat{P^2}=12$.
\item 
If $m>n=2$, then $\widehat{P^{2}}$ has a ${\bbk}$-basis 
$
\{\tau_{f_{i_1}}^{yx^2}, \tau_{f_{i_2}}^{xyx},
\tau_{g_{j_1}}^{y^2 x}, \tau_{g_{j_{2}}}^{yxy}, 
\tau_{g_{j_3}}^{x^{1+m}} 
\mid 1\leq i_{1}, i_{2}, i_{3} \leq m, 1\leq j_{1},j_{2} \leq 2 \}$,
\textcolor{red}{hence} $\dim_{\bbk} \widehat{P^2}=2m+6$.
\item If $n \geq 3$, then $\widehat{P^{2}}$ has a ${\bbk}$-basis
$
\{\tau_{f_{i_1}}^{yx^2}, \tau_{f_{i_2}}^{xyx}, 
\tau_{g_{j_1}}^{y^2 x}, \tau_{g_{j_2}}^{yxy}
\mid 1\leq i_{1}, i_{2} \leq m, 1\leq j_{1}, j_{2} \leq n \}
$,
\textcolor{red}{hence} $\dim_{\bbk} \widehat{P^2}=2(n+m)$.
\end{enumerate}
\end{lem}

\begin{proof}
\textcolor{red}{This claim can be proven in the same way as the proof of \cite[Lemma 2.2]{IU}.}
For $r=0,1,2$, there exists a ${\bbk}$-vector space isomorphism 
$\widehat{P^{r}} \cong \bigoplus_{h \in \cG^{r}}\Hom_{(\nabla A)^{\mathrm{e}}}(\nabla A s(h)\otimes_{\bbk} t(h)\nabla A, \nabla A) \cong \bigoplus_{h \in \cG^{r}} s(h)\nabla A t(h)$
 via the correspondence $\tau \in \widehat{P^{r}}$ to $\sum_{h \in \cG^{r}}\tau(s(h)\otimes t(h)) \in \bigoplus_{h \in \cG^{r}} s(h)\nabla A t(h)$. 
\textcolor{red}{Therefore, the claim follows. }
\end{proof}
The $2(n+m) \times 3(n+m)$ matrix $L_{1}$ \textcolor{red}{is defined} as follows:
{\tiny
\arraycolsep=2pt
\[
\text{{\normalsize $L_{1}:=$}}
\begin{pNiceArray}{cccccccccccccccc:cccccccccccccc}[first-row,last-col]
 &&&&& &&&&& &&&&& \textcolor{magenta}{2m+n} &&&&& &&&&& &&&& &\\
\beta &&&& \beta && -\beta &&&& -\beta &&&&&& -\beta &&&&&&&& \beta &&&&&& \\
& \ddots &&&& \Ddots && \ddots &&&& \ddots &&&&&& \ddots &&&&&&&& \ddots &&&&&\\
&& \ddots &&&& \Ddots && \ddots &&&& \ddots &&&&&& \ddots &&&&&&&& \ddots &&&&\\
&&& \ddots &&&& \ddots && \ddots &&&& \ddots &&&&&& \ddots &&&&&&&& \ddots &&&\\
&&&& \ddots &&&& \ddots && \ddots &&&& \ddots &&&&&& \ddots &&&&&&&& \ddots &&\\
&&&&& \beta &&&& \beta && \ddots &&&& -\beta &&&&&& \ddots &&&&&&&& \beta & \textcolor{magenta}{m}\\
\hdottedline
\beta &&&&&&&&&&&& \ddots &&&& -\beta &&&& \beta &&& \ddots &&& \beta &&&& \\
& \ddots &&&&&&&&&&&& \ddots &&&& \ddots &&&& \ddots &&& \ddots &&& \ddots &&& \\
&& \ddots &&&&&&&&&&&& \ddots &&&& \ddots &&&& \ddots &&& \ddots &&& \ddots && \\
&&& \beta &&&&&&&&&&&& -\beta &&&& -\beta &&&& \beta &&& -\beta &&& \beta &\textcolor{magenta}{m+n}\\
\hdottedline
\alpha &&&&&& -\alpha &&&&&&&&&& -\alpha &&&& \alpha &&&&&&&&&& \\
&\ddots &&&&&& \ddots &&&&&&&&&& \ddots &&&& \ddots &&&&&&&&& \\
&&\ddots &&&&&& \ddots &&&&&&&&&& \ddots &&&& \ddots &&&&&&&& \\
&&&\ddots &&&&&& \ddots &&&&&&&&&& \ddots &&&& \ddots &&&&&&& \textcolor{magenta}{m+2n}\\
\hdottedline
&&&&\ddots &&&&&& \ddots &&&&&&&&&& \ddots &&&& \ddots &&&&&& \\
&&&&&\ddots &&&&&& \ddots &&&&&&&&&& \ddots &&&& \ddots &&&&& \\
&&&&&&\ddots &&&&&& \ddots &&&&&&&&&& \ddots &&&& \ddots &&&& \\
&&&&&&&\ddots &&&&&& \ddots &&&&&&&&&& \ddots &&&& \ddots &&& \\
&&&&&&&&\ddots &&&&&& \ddots &&&&&&&&&& \ddots &&&& \ddots && \\
&&&&&&&&& \alpha &&&&&& -\alpha &&&&&&&&&& -\alpha &&&& \alpha &\\
\end{pNiceArray}
\]
}

\noindent
where \textcolor{red}{blank entries are} zeros. 
\textcolor{red}{If }$\alpha \neq 0$, 
\textcolor{red}{then this matrix  is denoted by $L_{1, 1}$. }
\textcolor{magenta}{In the special case $m=n=1$, $L_{1}$ takes the form}
$
L_{1}=
\left(\begin{array}{cccccc}
\beta & 0 & -\beta & -\beta & 0 & \beta \\
\beta & 0 & -\beta & -\beta & 0 & \beta \\
\alpha & -\alpha & 0 & -\alpha & \alpha & 0 \\
0 & \alpha & -\alpha & 0 & -\alpha & \alpha \\
\end{array}\right)
$. 
The $(m+2) \times (m+2)$ matrix $L_{2}$ is defined \textcolor{red}{as follows: }
{\large
\arraycolsep=2pt
\begin{align*}
\text{{\normalsize $L_{2}:=$}}
&\left(\begin{array}{ccccccccc}
1 & -\alpha & -\beta &&&&&& \\
& 1 & -\alpha & -\beta &&&&& \\
&& 1 & -\alpha & -\beta &&&& \\
&&&& \ddots &&&& \\
&&&&& 1 & -\alpha & -\beta & \\
&&&&&& 1 & -\alpha & -\beta \\
-\lambda_{2} & -\beta \lambda_{1} &&&&&& \beta \lambda_{m} & -\beta \lambda_{m+1} \\
\lambda_{1} &  \beta \lambda _{0} &&&&&& \lambda _{m+1} & -\lambda _{m+2} \\
\end{array}\right),\\
\end{align*}
}

\noindent
where \textcolor{red}{blank entries are} zeros. 
\textcolor{magenta}{In the special case $m=n=1$, $L_{2}$ takes the form}
\[
L_{2}=
\left(\begin{array}{ccc}
1 & -\alpha & -\beta \\
-\lambda _{2} & -\beta \lambda _{1} + \beta \lambda _{1} & -\beta \lambda _{2} \\
\lambda _{1} &  \beta \lambda _{0} + \lambda _{2} & -\lambda _{3} \\
\end{array}\right)
=
\left(\begin{array}{ccc}
1 & -\alpha & -\beta \\
-\lambda _{2} & 0 & -\beta \lambda _{2} \\
\lambda _{1} & \lambda _{2} & -\lambda _{3} \\
\end{array}\right).\\
\]

Next, we 
identify the matrix representation of $\widehat{\partial^2}$. 
\textcolor{red}{For that}, we recall the result by \cite{IU}.

\begin{prop}[{\cite[Lemma 2.4]{IU}}]
\label{par-IU}
\begin{enumerate}
\item 
\textcolor{red}{Assume} $n=1$ and $m=2$. 

\noindent
Let $\rho_1$ be the ordered basis 
$\{ \tau_{x_1},\, \ldots ,\, \tau_{x_5},\, \tau_{y_1},\, \ldots ,\, \tau_{y_{4}},\, \tau_{y_4}^{x^2},\, \ldots ,\, \tau_{y_{1}}^{x^2} \}$ for $\widehat{P^{1}}$, 
and \textcolor{red}{let} $\rho_2$ \textcolor{red}{be} the ordered basis
$\{
\tau_{f_1}^{yx^2},\, \tau_{f_2}^{yx^2},\, \tau_{g_1}^{y^2x},\, 
\tau_{g_1}^{yxy},\, \tau_{f_1}^{xyx},\, 
$
$
\tau_{f_2}^{xyx},\, 
\tau_{f_2}^{x^4},\, \tau_{f_1}^{x^4},\, \tau_{g_1}^{yx^3},\, \tau_{g_1}^{xyx^2},\, 
\tau_{g_1}^{x^5},\, \tau_{f_1}^{y^2},\, \tau_{f_2}^{y^2}
\}$
for $\widehat{P^{2}}$.
Then the matrix representation $M_2$ of $\widehat{\partial^2}$ with respect to $\rho_1$ and $\rho_2$ is \textcolor{red}{given by} 
$
\left(\begin{array}{cc}
\text{\large$L_{1}$} & \text{\large$0$} \\
\text{\large$0$} & \text{\large$L_{2}$} \\
0\cdots0 & 0\cdots0 \\
0\cdots0 & 0\cdots0 \\
0\cdots0 & 0\cdots0 \\
\end{array}\right). 
$
\item 
\textcolor{red}{Assume} $n=1$ and $m \geq 3$. 

\noindent
Let $\rho_1$ be the ordered basis 
$\{ \tau_{x_1},\, \ldots ,\, \tau_{x_{2m+1}},\, \tau_{y_1},\, \dots ,\, \tau_{y_{m+2}},\, \tau_{y_{m+2}}^{x^m},\, \ldots ,\, \tau_{y_1}^{x^m} \}$ 
for $\widehat{P^{1}}$, 
and \textcolor{red}{let} $\rho_2$ \textcolor{red}{be} the ordered basis
$\{
\tau_{f_1}^{yx^2},\, \ldots ,\, \tau_{f_m}^{yx^2},\, \tau_{g_1}^{y^2x},\, 
\tau_{g_1}^{yxy},\, \tau_{f_1}^{xyx},\, \ldots ,\, \tau_{f_m}^{xyx},\, 
\tau_{f_m}^{x^{m+2}},\, \ldots ,\, \tau_{f_1}^{x^{m+2}},\, \tau_{g_1}^{yx^{m+1}},\, \tau_{g_1}^{xyx^m} 
$
$
\tau_{g_1}^{x^{2m+1}}
\}$
for $\widehat{P^{2}}$.
Then the matrix representation $M_2$ of $\widehat{\partial^2}$ with respect to $\rho_1$ and $\rho_2$ is \textcolor{red}{given by} 
$
\left(\begin{array}{cc}
\text{\large$L_{1}$} & \text{\large$0$} \\
\text{\large$0$} & \text{\large$L_{2}$} \\
0\cdots0 & 0\cdots0 \\
\end{array}\right).
$
\end{enumerate}
\end{prop}
\begin{lem}\label{par-IM}
\begin{enumerate}
\item 
\textcolor{red}{Assume} $m=n=1$.

\noindent
Let $\rho_1$ be the ordered basis 
$\{ \tau_{x_1},\, \tau_{x_2},\, \tau_{x_3},\, \tau_{y_1},\, \tau_{y_2},\, \tau_{y_3},\, 
\tau_{y_3}^{x},\, \tau_{y_2}^{x}$, 
$\tau_{y_1}^{x},\, 
\tau_{x_1}^{y},\, \tau_{x_2}^{y},\, \tau_{x_3}^{y} 
\}$ 
 for $\widehat{P^{1}}$, 
and $\rho_2$ the ordered basis
$\{
\tau_{f_1}^{yx^2},\, \tau_{g_1}^{y^2x},\, \tau_{g_1}^{yxy},\, \tau_{f_1}^{xyx},\, 
\tau_{f_1}^{x^3},\, \tau_{g_1}^{yx^2},\, \tau_{g_1}^{xyx},\, \tau_{g_1}^{x^3},\, 
\tau_{g_1}^{y^3}$, 
$ \tau_{f_1}^{y^2x},\, \tau_{f_1}^{yxy},\, \tau_{f_1}^{y^3}
\}$
for $\widehat{P^{2}}$.
Then the matrix representation $M_2$ of $\widehat{\partial^2}$ with respect to $\rho_1$ and $\rho_2$ is \textcolor{red}{given by} 
$
\left(\begin{array}{ccc}
\text{\large$L_{1}$} & \text{\large$0$} & \text{\large$0$} \\
\text{\large$0$} & \text{\large$L_{2}$} & \text{\large$0$} \\
0\cdots0 & 0\cdots0 & 0\cdots0 \\
\text{\large$0$} & \text{\large$0$} & \text{\large$L_{2}$} \\
0\cdots0 & 0\cdots0 & 0\cdots0 \\
\end{array}\right).
$
\item 
\textcolor{red}{Assume} $n=2$.

\noindent
Let $\rho_1$ be the ordered basis 
$\{ \tau_{x_1}, \dots , \tau_{x_{n+2m}}, \tau_{y_1}, \dots , \tau_{y_{2n+m}} \}$ 
for $\widehat{P^{1}}$, 
and \textcolor{red}{let} $\rho_2$ \textcolor{red}{be} the ordered basis
$
\{\tau_{f_1}^{yx^2}, \dots ,\tau_{f_m}^{yx^2}, 
\tau_{g_1}^{y^2 x}, \tau_{g_2}^{y^2 x}, 
\tau_{g_1}^{yxy}, \tau_{g_2}^{yxy}, 
\tau_{f_1}^{xyx}, \dots , \tau_{f_m}^{xyx} 
\tau_{g_1}^{x^{1+m}}, \tau_{g_2}^{x^{1+m}} \}
$
for $\widehat{P^{2}}$.
Then the matrix representation $M_2$ of $\widehat{\partial^2}$ with respect to $\rho_1$ and $\rho_2$ is \textcolor{red}{given by} 
$
\left(\begin{array}{ccc}
\text{\large$L_{1}$}\\
0\cdots0\\
0\cdots0\\
\end{array}\right)
$.
\item\label{computecase}
\textcolor{red}{Assume} $n \geq 3$. 

\noindent
Let $\rho_1$ be the ordered basis 
$\{ \tau_{x_1},\, \ldots,\, \tau_{x_{n+2m}},\, \tau_{y_1},\, \ldots,\, \tau_{y_{2n+m}} \}$ 
for $\widehat{P^{1}}$, 
and \textcolor{red}{let} $\rho_2$ \textcolor{red}{be} the ordered basis
$
\{\tau_{f_1}^{yx^2},\, \ldots ,\,\tau_{f_m}^{yx^2},\, 
\tau_{g_1}^{y^2 x},\, \ldots,\, \tau_{g_n}^{y^2 x},\, 
\tau_{g_1}^{yxy},\, \ldots,\, \tau_{g_n}^{yxy},\, 
\tau_{f_1}^{xyx},\, \ldots,\, \tau_{f_m}^{xyx} \}
$
for $\widehat{P^{2}}$.
Then the matrix representation $M_2$ of $\widehat{\partial^2}$with respect to $\rho_1$ and $\rho_2$ is \textcolor{red}{given by} $L_{1}$. 
\end{enumerate}
\end{lem}
\begin{proof}
These \textcolor{red}{claims} 
\textcolor{red}{can be proved in} the similar calculation 
in \cite[Lemma 2.4]{IU}. 
We only compute the first and second columns of \eqref{computecase}.
\textcolor{blue}{The others are proved in a similar way.}
For $h \in \cG^{2}$, we have
\[
\widehat{\partial^{2}}(\tau_{x_{1}})(s(h)\otimes t(h))
=(\tau_{x_{1}} \circ \partial^{2})(s(h)\otimes t(h))
\]
\begin{align*}
&=
\begin{cases}
\tau_{x_{1}}(
&\hspace{-10pt}(s(x_{i}) \otimes t(x_{i}))x_{i+n}y_{i+2n}
+x_{i}(s(x_{i+n}) \otimes t(x_{i+n}))y_{i+2n}\\
&\quad\quad+x_{i}x_{i+n}(s(y_{i+2n}) \otimes t(y_{i+2n}))
)\\
&\hspace{-10pt}-\alpha(
(s(x_{i}) \otimes t(x_{i}))y_{i+n}x_{i+n+m}
+x_{i}(s(y_{i+n}) \otimes t(y_{i+n}))x_{i+n+m}\\
&\quad\quad+x_{i}y_{i+n}(s(x_{i+n+m}) \otimes t(x_{i+n+m}))
)\\
&\hspace{-10pt}-\beta(
(s(y_{i}) \otimes t(y_{i}))x_{i+m}x_{i+n+m}
+y_{i}(s(x_{i+m}) \otimes t(x_{i+m}))x_{i+n+m}\\
&\quad\quad+y_{i}x_{i+m}(s(x_{i+n+m}) \otimes t(x_{i+n+m}))
)\quad\text{if $h=f_{i}$ for $1 \leq i \leq m$,}\\
\tau_{x_{1}}(
&\hspace{-10pt}(s(x_{j}) \otimes t(x_{j}))y_{j+n}y_{j+n+m}
+x_{j}(s(y_{j+n}) \otimes t(y_{j+n}))y_{j+n+m}\\
&\quad\quad+x_{j}y_{j+n}(s(y_{j+n+m}) \otimes t(y_{j+n+m}))
)\\
&\hspace{-10pt}-\alpha(
(s(y_{j}) \otimes t(y_{j}))x_{j+m}y_{j+n+m}
+y_{j}(s(x_{j+m}) \otimes t(x_{j+m}))y_{j+n+m}\\
&\quad\quad+y_{j}x_{j+m}(s(y_{j+n+m}) \otimes t(y_{j+n+m}))
)\\
&\hspace{-10pt}-\beta(
(s(y_{j}) \otimes t(y_{j}))y_{j+m}x_{j+2m}
+y_{j}(s(y_{j+m}) \otimes t(y_{j+m}))x_{j+2m}\\
&\quad\quad+y_{j}y_{j+m}(s(x_{j+2m}) \otimes t(x_{j+2m}))
)\quad\quad\text{if $h=g_{j}$ for $1 \leq j \leq n$;}\\
\end{cases}\\
&=
\begin{cases}
\beta y_{1}x_{m+1}x_{n+m+1}
& \text{if $h=f_{1}$,}\\
0 & \text{if $h=f_{i}$ for $2 \leq i \leq m$,}\\
\alpha y_{1}x_{m+1}y_{n+m+1}+ \beta y_{1}y_{m+1}x_{2m+1}
& \text{if $h=g_{1}$,}\\
0 & \text{if $h=g_{j}$ for $2 \leq j \leq n$;}\\
\end{cases}\\
&=\beta \tau_{f_{1}}^{yx^{2}} +\beta \tau_{g_{1}}^{y^{2}x} +\alpha \tau_{g_{1}}^{yxy}.
\end{align*}
\textcolor{red}{Hence we} obtain the first column of $L_{1}$.

Since $n \geq 3$, for $h \in \cG^{2}$, we have
\begin{align*}
\widehat{\partial^{2}}(\tau_{x_{2}})(s(h)\otimes t(h))
&=(\tau_{x_{2}} \circ \partial^{2})(s(h)\otimes t(h))\\
&=
\begin{cases}
x_{2}x_{n+2}y_{2n+2}-\alpha x_{2}y_{n+2}x_{n+m+2}
& \text{if $h=f_{2}$,}\\
0 & \text{if $h=f_{i}$ for $i=1$ or $3 \leq i \leq m$,}\\
x_{2}y_{n+2}y_{n+m+2}
& \text{if $h=g_{2}$,}\\
0 & \text{if $h=g_{j}$ for $j=1$ or $3 \leq j \leq n$;}\\
\end{cases}\\
&=
\begin{cases}
\beta y_{2}x_{m+2}x_{n+m+2}
& \text{if $h=f_{2}$,}\\
0 & \text{if $h=f_{i}$ for $i=1$ or $3 \leq i \leq m$,}\\
\alpha y_{2}x_{m+2}y_{n+m+2}+ \beta y_{2}y_{m+2}x_{2m+2}
& \text{if $h=g_{2}$,}\\
0 & \text{if $h=g_{j}$ for $j=1$ or $3 \leq j \leq n$;}\\
\end{cases}\\
&=\beta \tau_{f_{2}}^{yx^{2}} +\beta \tau_{g_{2}}^{y^{2}x} +\alpha \tau_{g_{2}}^{yxy}. 
\end{align*}
\textcolor{red}{Hence we} obtain the second column of \textcolor{magenta}{$L_{1}$}. 
\end{proof}
\subsection{The rank of the representation \textcolor{blue}{matrices} $M_{2}$}
\label{subsec-3-2}
\textcolor{red}{To} compute the rank of $M_{2}$, 
we recall 
\textcolor{red}{the definition of the property of a circulant matrix}.
\begin{dfn}[see \cite{I}]
For $r \geq 1$, let $C=[c_{p,q}]_{1\leq p,q \leq r}$ be an $r$-th square matrix.
If there exist $a_{0}, a_{1}, \cdots , a_{r-1}$ 
such that $a_{k}=c_{p,q}$ for $q-p \equiv k\,\,(\textcolor{magenta}{\mathrm{mod}}\, r)$, 
then $C$ is called a \textit{circulant matrix} 
\textcolor{red}{ and is denoted by $C=C(a_{0}, \ldots , a_{r-1})$.}
\textcolor{red}{For a circulant matrix $C=C(a_{0}, \ldots , a_{r-1})$,} the polynomial $f(x)=\Dsum_{k=0}^{r-1}a_{k}x^{k}$ is called the \textit{associated polynomial of $C$}.
\end{dfn}
\textcolor{magenta}{We compute the rank of matrices as examples of \cite{KT} and \cite{I}.}
\begin{lem}\label{rank-C}
Let $C=C(1,0,\ldots 0,1,0,\ldots, 0)$ be a circulant matrix with the associated polynomial $f(x)=x^{m}+1$, 
and \textcolor{red}{let} $C_{n+m-1}$ \textcolor{red}{be} the principal submatrix of $C$ deleting 
the last row and column\textcolor{red}{,} where $C$ is the circulant matrix with associated polynomial $f(x)=x^{m}+1$, that is,
\[
C=
\text{{\small
$
\begin{pNiceArray}{cccc:ccc}[first-row, last-col]
  &&& \textcolor{magenta}{m} &&&& \\
  1 &&&& 1 &&& \\
  & \ddots&&&& \ddots && \\
  && \ddots &&&& 1 & \textcolor{magenta}{n} \\
  \hdottedline
  1 &&& \ddots &&&& \\
  & \ddots &&& \ddots &&& \\
  && \ddots &&& \ddots && \\
  &&& 1 &&& 1&
\end{pNiceArray}
$
}},\,
C_{n+m-1}=
\text{{\tiny
$
\begin{pNiceArray}{cccc:c:ccc}[first-row,last-col]
  &&& \textcolor{magenta}{n}& \textcolor{magenta}{n+1}&&&&\\
  1&&&&&1&&&\\
  &\ddots&&&&&\ddots&&\\
  &&\ddots&&&&&1&\textcolor{magenta}{n}\\
  \hdottedline
  &&&\ddots&&&&&\textcolor{magenta}{n+1}\\
  \hdottedline
  1&&&&\ddots&&&&\\
  &\ddots&&&&\ddots&&&\\
  &&\ddots &&&&\ddots&&\\
  &&&1 &&&&1&
\end{pNiceArray}
$
}}.
\]
Then \textcolor{red}{we have the following\textup{:} }
\begin{enumerate}
\item
$
\rank C
=
\begin{cases}
(n+m)-1 & \text{\textup{if $n$ and $m$ are odd,}}\\
(n+m)  &\text{\textup{otherwise.}}\\
\end{cases}
$\label{rank-C-equ1}
\item
$
\rank C_{n+m-1}=n+m-1\textcolor{red}{.}
$\label{rank-C-equ2}
\end{enumerate}
\end{lem}
\begin{proof}
\eqref{rank-C-equ1}
By \cite[Proposition 1.1]{I}, the rank of $C$ is equal to $(n+m)- d$\textcolor{red}{,}
where $d$ is the degree of $\gcd(x^{(n+m)}-1, f(x))$.
Since $\gcd(n, m)=1$, we have $\gcd(x^{2m}-1, x^{2n}-1)=(x^{2}-1)$. 
\textcolor{red}{Then} 
$
\gcd(x^{(n+m)}-1, f(x))=
\gcd(x^{(n+m)}-1, x^{m}+1)=
\begin{cases}
(x+1) & \text{if $n$ and $m$ are odd,}\\
1  &\text{otherwise.}\\
\end{cases}
$
Therefore, we obtain
\[  
\rank C
=
\begin{cases}
(n+m)-\deg(x+1) & \text{if $n$ and $m$ are odd,}\\
(n+m)-\deg(1)  &\text{otherwise,}\\
\end{cases}
=
\begin{cases}
(n+m)-1 & \text{if $n$ and $m$ are odd,}\\
(n+m)  &\text{otherwise.}\\
\end{cases}
\]
\eqref{rank-C-equ2}
Let $p(t)$ be the characteristic polynomial of $C$, then $\dfrac{1}{n}p'(t)$ is the characteristic polynomial of $C_{n+m-1}$ by \cite[proof of Theorm 2.1]{KT}. 
Let $\zeta$ be an $(n+m)$-th primitive root of unity, then 
\textcolor{red}{$\zeta$} satisfies $(z^{n+m}-1)=\prod_{r=0}^{n+m-1}(z-\zeta^{r})$. 
We have
$p(t)=\det (tE_{n+m}-C)=\prod_{r=0}^{n+m-1}(t-1-\zeta^{r})=((t-1)^{n+m}-1)$. 
\textcolor{red}{Then} 
$p'(t)=(n+m)(t-1)^{n+m-1}$. 
\textcolor{red}{Hence} $\textcolor{magenta}{\det(C_{n+m-1})}=\dfrac{1}{n}p'(0)=(-1)^{n+m-1}(n+m)\neq 0$. 
\textcolor{red}{Therefore,} $\rank C_{n+m-1} = n+m-1$.
\end{proof}
\begin{lem}\label{rank-L1}
If $m >n>1$ and $gcd(n, m) = 1$, 
then 
$
\rank\,L_{1}=
\begin{cases}
n+m-1 \quad \text{\textup{if $n+m$ is even and $\alpha=0$}}, \\
n+m \quad \text{\textup{otherwise}}.
\end{cases}
$
\end{lem}
\begin{proof}
\textcolor{red}{Remark} that 
$
\rank L_{1} = \rank {L_{1}}^{\mathrm{T}}.
$
\textcolor{red}{We} have 
\begin{equation}\label{row-op}
\text{{\normalsize $L_{1}^{\mathrm{T}}$}}
\rightarrow
\begin{pNiceArray}{c:c}
\Block{1-1}<\Large>{\beta C} & \Block{1-1}<\Large>{\alpha E_{n+m}} \\
\hdottedline
\Block{1-1}<\Large>{$0$}& \Block{1-1}<\Large>{$0$} \\
\end{pNiceArray}
\textcolor{red}{,}
\end{equation}
\noindent
where \textcolor{red}{$C=C(1,0,\ldots 0,1,0,\ldots, 0)$} is the circulant matrix with the associated polynomial $f(x)=x^{m}+1$, 
and $E_{n+m}$ is the $(n+m)$-th identity matrix. 
\textcolor{red}{Therefore, by Lemma \ref{rank-C} \eqref{rank-C-equ1}, the conclusion follows. }
\textcolor{blue}{Details of the row operations used in \eqref{row-op} are given below. }

\end{proof}
\begin{landscape}
\textcolor{blue}{The details of the row operations in \eqref{row-op}.}
\quad $L_{1}^{\mathrm{T}} \rightarrow$
\vspace{-3pt}
{\tiny
\arraycolsep=2pt
\begin{align*}
&\begin{pNiceArray}{cccccc:cccc:cccc:cccccc}[first-row, last-col]
&&&&&& \textcolor{magenta}{\makebox[4pt][l]{m}} &&&& \textcolor{magenta}{\makebox[4pt][l]{m+n}} &&&& \textcolor{magenta}{\makebox[4pt][l]{m+2n}} &&&&&&\\
\beta &&&&&& \beta &&&& \alpha&&&&&&&&&& \\
& \ddots &&&&&& \ddots && && \ddots &&&&&&&&& \\
&& \ddots&&&&&& \ddots& &&& \ddots &&&&&&&& \\
&&& \ddots &&&&&& \beta &&&& \alpha&&&&&&& \textcolor{magenta}{n}\\[-5pt]
\hdottedline
\beta &&&& \ddots &&&&& &&&&& \alpha&&&&&& \\
& \ddots &&&& \beta &&&& &&&&&& \ddots &&&&& \\
-\beta && \ddots &&&&&&& & -\alpha&&&&&& \ddots &&&& \\
& \ddots && \ddots &&&&&& && \ddots &&&&&& \ddots &&& \\
&& \ddots && \ddots &&&&& &&& \ddots &&&&&& \ddots && \\
&&& \ddots && \beta &&&& &&&& -\alpha&&&&&& \alpha&\\
-\beta  &&&& \ddots &&&&& &&&&& -\alpha&&&&&& \\
& \ddots  &&&& -\beta &&&& &&&&&& \ddots &&&&& \textcolor{magenta}{2m}\\[-3pt]
\hdottedline
&& \ddots  &&&& -\beta &&& &&&&&&& \ddots &&&& \\
&&&\ddots  &&&& \ddots && &&&&&&&& \ddots &&& \\
&&&&\ddots  &&&& \ddots & &&&&&&&&& \ddots && \\
&&&&& -\beta  &&&& -\beta  &&&&&&&&&& -\alpha& \textcolor{magenta}{2m+n}\\
\hdottedline
-\beta &&&&&& -\beta  &&& & -\alpha&&&&&&&&& &\\
& \ddots &&&&&& \ddots  && && \ddots &&&&&&&&& \\
&& \ddots &&&&&& \ddots  & &&& \ddots &&&&&&&& \\
&&& \ddots &&&&&& -\beta &&&& -\alpha&&&&&&& \textcolor{magenta}{2m+2n}\\[-3pt]
\hdottedline
&&&& \ddots && \beta &&& & \alpha&&&& -\alpha&&&&&& \\
&&&&& \ddots && \ddots && && \ddots &&&& \ddots &&&&& \\
&&&&&& \ddots && \ddots & &&& \ddots &&&& \ddots &&&& \\
&&&&&&& \ddots && \beta &&&& \alpha&&&& \ddots &&& \\
\beta &&&&&&&& \ddots & &&&&& \alpha&&&& \ddots && \\
& \ddots &&&&&&&& -\beta &&&&&& \ddots &&&& -\alpha& \textcolor{magenta}{3m+2n}\\[-3pt]
\hdottedline
&& \ddots &&&& \beta &&& &&&&&&& \ddots &&&& \\
&&& \ddots &&&& \ddots && &&&&&&&& \ddots &&& \\
&&&& \ddots &&&& \ddots & &&&&&&&&& \ddots && \\
&&&&& \beta &&&& \beta &&&&&&&&&& \alpha &\\
\end{pNiceArray}
\hspace{-15pt}
\text{{\normalsize $\rightarrow$}}
&&\begin{pNiceArray}{cccccc:cccc:cccc:cccccc}[first-row, last-col]
&&&&&& \textcolor{magenta}{\makebox[4pt][l]{m}} &&&& \textcolor{magenta}{\makebox[4pt][l]{m+n}} &&&& \textcolor{magenta}{\makebox[4pt][l]{m+2n}} &&&&&&\\
\beta &&&&&& \beta &&&& \alpha &&&&&&&&&& \\
& \ddots &&&&&& \ddots && && \ddots &&&&&&&&& \\
&& \ddots&&&&&& \ddots& &&& \ddots &&&&&&&& \\
&&& \ddots &&&&&& \beta &&&& \alpha &&&&&&&\textcolor{magenta}{n}\\[-5pt]
\hdottedline
\beta &&&& \ddots &&&&& &&&&& \alpha &&&&&& \\
& \ddots &&&& \beta &&&& &&&&&& \ddots &&&&& \\
-\beta && \ddots &&&&&&& & -\alpha &&&&&& \ddots &&&& \\
& \ddots && \ddots &&&&&& && \ddots &&&&&& \ddots &&& \\
&& \ddots && \ddots &&&&& &&& \ddots &&&&&& \ddots && \\
&&& \ddots && \beta &&&& &&&& -\alpha &&&&&& \alpha &\\
-\beta  &&&& \ddots &&&&& &&&&& -\alpha &&&&&& \\
& \ddots  &&&& -\beta &&&& &&&&&& \ddots &&&&&\textcolor{magenta}{2m}\\[-3pt]
\hdottedline
&& \ddots  &&&& -\beta &&& &&&&&&& \ddots &&&& \\
&&&\ddots  &&&& \ddots && &&&&&&&& \ddots &&& \\
&&&&\ddots  &&&& \ddots & &&&&&&&&& \ddots && \\
&&&&& -\beta  &&&& -\beta  &&&&&&&&&& -\alpha &\textcolor{magenta}{2m+n}\\
\hdottedline
\rule{0pt}{38pt}
\Block{1-6}<\huge>{$0$}&&&&&& \Block{1-4}<\huge>{$0$} &&&& \Block{1-4}<\huge>{$0$} &&&& \Block{1-6}<\huge>{$0$}&&&&&&\\
\rule{0pt}{13pt} &&&&&& &&&& &&&& &&&&&&\textcolor{magenta}{2m+2n}\\[-3pt]
\hdottedline
\rule{0pt}{18pt}
&&&& -\beta && \beta &&& & \alpha &&&& -\alpha &&&&&& \\
&&&&& \ddots && \ddots && && \ddots &&&& \ddots &&&&& \\
&&&&&& \ddots && \ddots & &&& \ddots &&&& \ddots &&&& \\
&&&&&&& \ddots && \beta &&&& \alpha &&&& \ddots &&& \\
\rule{0pt}{20pt}
\Block[B]{2-2}{\text{\arraycolsep=1pt
  $\begin{array}{ccc} 
    \beta && \\[-3pt]
    & \ddots & \\[-3pt]
    && \beta 
  \end{array}$}} 
  &&&&&&&& \ddots & &&&&& 
\Block[B]{2-2}{\text{\arraycolsep=1pt
  $\begin{array}{ccc} 
    \alpha && \\[-3pt]
    & \ddots & \\[-3pt] 
    && \alpha 
  \end{array}$}} 
  &&&& \ddots && \\
&&&&&&&&& -\beta &&&&&&&&&& -\alpha &\textcolor{magenta}{3m+2n}\\[-3pt]
\hdottedline
\rule{0pt}{38pt}
\Block{1-6}<\huge>{$0$}&&&&&& \Block{1-4}<\huge>{$0$} &&&& \Block{1-4}<\huge>{$0$} &&&& \Block{1-6}<\huge>{$0$}&&&&&&\\
\rule{0pt}{13pt}\\
\end{pNiceArray}
\end{align*}
}
\end{landscape}
\begin{landscape}
{\tiny
\arraycolsep=2pt
\begin{align*}
&\text{{\normalsize $\rightarrow$}}
\begin{pNiceArray}{cccccc:cccc:cccc:cccccc}[first-row, last-col]
&&&&&& \textcolor{magenta}{\makebox[4pt][l]{m}} &&&& \textcolor{magenta}{\makebox[4pt][l]{m+n}} &&&& \textcolor{magenta}{\makebox[4pt][l]{m+2n}} &&&&&&\\
\beta &&&&&& \beta &&&& \alpha &&&&& &&&&& \\
& \ddots &&&&&& \ddots &&&& \ddots &&&&&&&&& \\
&& \ddots&&&&&& \ddots&&&& \ddots &&&&&&&& \\
&&& \ddots &&&&&& \beta &&&& \alpha &&&&&&& \textcolor{magenta}{n}\\
\hdottedline
\beta &&&& \ddots &&&&&&&&&& \alpha &&&&&& \\
& \ddots &&&& \beta &&&&&&&&&& \ddots &&&&& \\
&& \ddots &&&& \beta &&&&&&&&&& \ddots &&&& \\
&&& \ddots &&&& \ddots &&&&&&&&&& \ddots &&& \\
&&&& \ddots &&&& \ddots &&&&&&&&&& \ddots && \\
&&&&& \beta &&&& \beta &&&&&&&&&& \alpha & \textcolor{magenta}{m+n}\\
\hdottedline
\Block{1-6}<\huge>{$0$}&&&&&& \Block{1-4}<\huge>{$0$} &&&& \Block{1-4}<\huge>{$0$} &&&& \Block{1-6}<\huge>{$0$}&&&&&& \textcolor{magenta}{2m}\\
\hdottedline
&& -\beta  &&&& -\beta &&&&& &&&&& -\alpha &&&& \\
&&& \ddots &&&& \ddots &&&&& &&&&& \ddots &&& \\
&&&& \ddots &&&& \ddots &&&&& &&&&& \ddots && \\
&&&&& -\beta  &&&& -\beta  &&&&& &&&&& -\alpha & \textcolor{magenta}{2m+n}\\
\hdottedline
\Block{1-6}<\huge>{$0$}&&&&&& \Block{1-4}<\huge>{$0$} &&&& \Block{1-4}<\huge>{$0$} &&&& \Block{1-6}<\huge>{$0$}&&&&&& \textcolor{magenta}{2m+2n}\\
\hdottedline
-\beta &&&& -\beta &&&&&&&&&& -\alpha &&&&&& \\
& \ddots &&&& \ddots &&&&&&&&&& \ddots &&&&& \\
&& \ddots &&&& \ddots &&&&&&&&&& \ddots &&&& \\
&&& \ddots &&&& \ddots &&&&&&&&&& \ddots &&& \\
&&&& \ddots &&&& \ddots &&&&&&&&&& \ddots && \\
&&&&& -\beta &&&& \beta &&&&&&&&&& -\alpha & \textcolor{magenta}{3m+2n}\\
\hdottedline
\Block{1-6}<\huge>{$0$}&&&&&& \Block{1-4}<\huge>{$0$} &&&& \Block{1-4}<\huge>{$0$} &&&& \Block{1-6}<\huge>{$0$}&&&&&&\\
\end{pNiceArray}
\hspace{-12pt}
\text{{\normalsize $\rightarrow$}}
\begin{pNiceArray}{cccccc:cccc:cccc:cccccc}[first-row, last-col]
&&&&&& \textcolor{magenta}{\makebox[4pt][l]{m}} &&&& \textcolor{magenta}{\makebox[4pt][l]{m+n}} &&&& \textcolor{magenta}{\makebox[4pt][l]{m+2n}} &&&&&&\\
\beta &&&&&& \beta &&&& \alpha &&&&& &&&&& \\
& \ddots &&&&&& \ddots &&&& \ddots &&&&&&&&& \\
&& \ddots&&&&&& \ddots&&&& \ddots &&&&&&&& \\
&&& \ddots &&&&&& \beta &&&& \alpha &&&&&&& \textcolor{magenta}{n}\\
\hdottedline
\beta &&&& \ddots &&&&&&&&&& \alpha &&&&&& \\
& \ddots &&&& \beta &&&&&&&&&& \ddots &&&&& \\
&& \ddots &&&& \beta &&&&&&&&&& \ddots &&&& \\
&&& \ddots &&&& \ddots &&&&&&&&&& \ddots &&& \\
&&&& \ddots &&&& \ddots &&&&&&&&&& \ddots && \\
&&&&& \beta &&&& \beta &&&&&&&&&& \alpha &\textcolor{magenta}{n+m}\\
\hdottedline
\Block{1-6}<\huge>{$0$}&&&&&& \Block{1-4}<\huge>{$0$} &&&& \Block{1-4}<\huge>{$0$} &&&& \Block{1-6}<\huge>{$0$}&&&&&&\\
\end{pNiceArray}\\
&\text{{\normalsize $=$}}
\begin{pNiceArray}{c:c}
\Block{1-1}<\Large>{\beta C} & \Block{1-1}<\Large>{\alpha E_{n+m}} \\
\hdottedline
\Block{1-1}<\Large>{$0$}& \Block{1-1}<\Large>{$0$} \\
\end{pNiceArray}.
\end{align*}
}
\hspace{150pt} 
\textcolor{blue}{Hence \eqref{row-op} is obtained by elementary row operations.}
\end{landscape}
\clearpage
%
\subsection{Main results}
\label{subsec-3-3}
\textcolor{blue}{By a computation similar to that in}
\cite{IU}, we obtain the \textcolor{red}{dimension} formula of $\HH^{r}(\nabla A)$ for $r\geq 0$. 

\begin{thm}\label{HHdim-IM}
Let $A=A(\alpha, \beta)$ be a graded down-up algebra with weights $\mathrm{deg}\ x=n$, $\mathrm{deg}\ y=m$ and  $\beta\neq0$ 
where $m \geq n >1$ and $\mathrm{gcd}(n, m)=1$, and $\nabla A$ \textcolor{red}{its} Beilinson algebra. 
Then we obtain the \textcolor{red}{dimension} formula of the Hochschild cohomology groups $\HH^{r}(\nabla A)$ 
as follows\textup{:}
\begin{itemize}
\item $\dim_{\bbk} \HH^0(\nabla A)=1;$
\item 
$\dim_{\bbk} \HH^1(\nabla A)=
\begin{cases}
2 \quad \text{\textup{for Case I}},\\
1 \quad \text{\textup{for Case I\hspace{-1.2pt}I}};
\end{cases}$
\item $\dim_{\bbk} \HH^2(\nabla A)=
\begin{cases}
m+5 \quad \text{\textup{for $n=2$ and Case I}},\\
m+4 \quad \text{\textup{for $n=2$ and Case  I\hspace{-1.2pt}I}},\\
n+m+1 \quad  \text{\textup{for $n \geq 3$ and Case I}},\\
n+m \quad \text{\textup{for $n \geq 3$ and Case I\hspace{-1.2pt}I}};
\end{cases}$
\item $\dim_{\bbk} \HH^{r}(\nabla A)= 0$ for ${r}\geq 3$.
\end{itemize}
\end{thm}
\begin{proof}
We compute the dimensions of $\HH^{r}(\nabla A)$ 
\textcolor{blue}{in a similar way as in}
\cite[Proof of Theorem 1.4]{IU}. 

Since $\dim_{\bbk} \HH_{0}(\nabla A)$ \textcolor{blue}{equals $1$}, we obtain
$\dim_{\bbk} \Ker \widehat{\partial^{1}}= \dim_{\bbk}\HH_{0}(\nabla A) =1$. 
\textcolor{red}{Also, we have}
\[
\dim_{\bbk} \Image \widehat{\partial^{1}}
 = \dim_{\bbk} \widehat{P^{0}} - \dim_{\bbk} \Ker \widehat{\partial^{1}}
=2(n+m)-1.
\]

\textcolor{red}{By} Lemmas \ref{bas-IM}, \ref{par-IM} \textcolor{red}{and} \ref{rank-L1}, 
\textcolor{red}{we have}

\begin{align*}
\dim_{\bbk} \HH^{1}(\nabla A) 
&= \dim_{\bbk} \Ker \widehat{\partial^{2}} - \dim_{\bbk} \Image \widehat{\partial^{1}}
= (\dim_{\bbk} \widehat{P^{1}} - \dim_{\bbk} \Image \widehat{\partial^{2}}) - \dim_{\bbk} \Image \widehat{\partial^{1}}\\
&= \dim_{\bbk} \widehat{P^{1}} - \rank M_{2} - \dim_{\bbk} \Image \widehat{\partial^{1}}
= \dim_{\bbk} \widehat{P^{1}} - \rank L_{1} - \dim_{\bbk} \Image \widehat{\partial^{1}}\\
&= 3(n+m) - \rank L_{1} - (2(n+m)-1)
= (n+m)+1 - \rank L_{1}. 
\end{align*}

\textcolor{red}{Hence} 
$\dim_{\bbk} \HH^{1}(\nabla A)=
\begin{cases}
2 \quad \text{\textup{for Case I}},\\
1 \quad \text{\textup{for Case I\hspace{-1.2pt}I}}.
\end{cases}
$

Also, \textcolor{red}{we have}
\begin{align*}
\dim_{\bbk} \HH^{2}(\nabla A) 
&= \dim_{\bbk} \Ker \widehat{\partial^{3}} - \dim_{\bbk} \Image \widehat{\partial^{2}}
= \dim_{\bbk} \widehat{P^{2}} - \dim_{\bbk} \Image \widehat{\partial^{2}}\\
&= \dim_{\bbk} \widehat{P^{2}} - \rank M_{2} 
= \dim_{\bbk} \widehat{P^{2}} - \rank L_{1} \\
&= \dim_{\bbk} \widehat{P^{2}} +\dim_{\bbk} \HH^{1}(\nabla A)  -((n+m)+1) \\
&= \dim_{\bbk} \widehat{P^{2}} +\dim_{\bbk} \HH^{1}(\nabla A)  -((n+m)+1) \\
&=
\begin{cases}
(2n+3m) + \dim_{\bbk} \HH^{1}(\nabla A)  + ((n+m)+1) \quad \text{if $n=2$},\\
2(n+m) + \dim_{\bbk} \HH^{1}(\nabla A)  + ((n+m)+1) \quad \text{if $n \geq 3$}. 
\end{cases}
\end{align*}

\textcolor{red}{Hence} 
$\dim_{\bbk} \HH^{2}(\nabla A)=
\begin{cases}
m+5 \quad \text{\textup{for $n=2$ and Case I}},\\
m+4 \quad \text{\textup{for $n=2$ and Case I\hspace{-1.2pt}I}},\\
n+m+1 \quad  \text{\textup{for $n \geq 3$ and Case I}},\\
n+m \quad \text{\textup{for \textcolor{magenta}{$n \geq 3$} and Case I\hspace{-1.2pt}I}}.
\end{cases}
$
\end{proof}
At the end of this section, we give a corollary of Theorem \ref{HHdim-IM} 
for the case 
$m > n > 1$\textcolor{red}{,}
 using the same discussion in \cite[Section 3]{IU}.

Regarding the bounded derived category of coherent sheaves on a smooth projective variety, Bondal--Polishchuk \cite{BP} proved the following theorem: 
\begin{thm}[{\cite[Lemma 3.1]{BP}}]\label{BP}
Let $\Db(\coh X)$ be the bounded derived category of coherent sheaves on a smooth projective variety $X$ and 
\textcolor{red}{let} $\kappa$ \textcolor{red}{be} the automorphism on $K_{0}(\Db(\coh X))$ induced by the Serre functor $K$ of $\Db(\coh X)$.
Then the action of $(-1)^{\mathrm{dim}\, X}\mathfrak{\kappa}$ on $K_{0}(\Db(\coh X))$ is unipotent. 
\end{thm}
To use Theorem \ref{BP} in order to check the correspondence between down-up algebras and some geometrical objects, we recall a useful matrix which is called the \textcolor{red}{{\it Gram matrix}}. 
\begin{dfn}[{see \cite{BP}}]
Let $\mathbf{T}$ be a triangulated category and 
\textcolor{red}{let} $K_{0}(\mathbf{T})$ \textcolor{red}{be} the Grothendieck group of $\mathbf{T}$.
The \textcolor{blue}{isomorphism class} of $X \in \mathbf{T}$ 
\textcolor{red}{is denoted by}
$[X]  \in K_{0}(\mathbf{T})$. 
Suppose that $\chi$ is the bilinear form on $K_{0}(\mathbf{T})$ defined by 
 \[
\chi([X], [Y]):=\sum_{{k} \in \mathbb{Z}}(-1)^{k}\dim_{\bbk} \Hom^{k}(X, Y).
\]
 If the category $\mathbf{T}$ is generated by an exceptional collection $\{E_{k}\}_{1 \leq {k} \leq r}$ of length $r$, then the Gram matrix $M$ of $\chi$ for this collection of vectors is the $r \times r$ \textcolor{blue}{matrix defined as} $M=(\chi([E_{p}], [E_{q}]))_{1 \leq {p}, {q} \leq r}$.
\end{dfn}
\begin{dfn}[{see \cite{ASS}}]
Let $B$ be a basic finite dimensional ${\bbk}$-algebra with a complete set $\{e_{1}, \dots , e_{r}\}$ of primitive orthogonal idempotents. 
The \textit{Cartan  matrix} of $B$ is defined by the $r \times r$ matrix  $C_{B}=(\dim_{\bbk}e_{p}Be_{q})_{1 \leq {p}, {q} \leq r}$.
\end{dfn}
Let $K_{0}(\Db(\tails A))$ be the Grothendieck group of the bounded derived category $\Db(\tails A)$
and 
\textcolor{red}{let} $\mathfrak{s}$ \textcolor{red}{be} the automorphism on $K_{0}(\Db(\tails A))$ induced by the Serre functor $S$ of $\Db(\tails A)$.
Note that, in this case, the existence of the Serre functor $\mathfrak{s}$ on $\Db(\tails A)$ is proven by de Naeghel--Van den Bergh \cite[Appendix A]{dNV}.
\begin{prop}[{see \cite[Section 3]{IU}}]\label{prop-gram}
Let $A=A(\alpha, \beta)$ be a down up algebra with weights 
$\deg\, x =n$, $\deg\, y=m$, $\beta \neq 0$, 
\textcolor{red}{let $\nabla A$ be its the Beilinson and let} 
$\Db(\tails A)$ be the bounded derived category of $\tails A$. 
Then 
$\{A(k) \mid 0 \leq {k} \leq 2(n+m)-1 \}$ is a full strong exceptional collection in $\Db(\tails A)$ and
the Gram matrix $M$ of $\chi$ on $\Db(\tails A)$ for this collection of vectors is equal to the Cartan matrix of $\nabla A$.
\end{prop}
\begin{proof}
By \cite[Propositions 4.3 and 4.4]{MM}, $\{A(k) \mid 0 \leq {k} \leq 2(n+m)-1 \}$ is a full strong exceptional collection of length $2(n+m)$.
Then the Gram matrix $M$ of $\chi$ for this collection of vectors satisfies 
\noindent
$M=(\dim_{\bbk} \Hom(A({p}-1), A({q}-1))_{1 \leq {p}, {q} \leq 2(n+m)}$ (see \cite[Section 3]{BP}).
Therefore, \textcolor{red}{the claim is proven.} 
\end{proof}
\begin{cor}\label{cor-exiPV}
Let $A=A(\alpha, \beta)$ be a down up algebra with weights 
$\deg\, x =n$, $\deg\, y=m$, $\beta \neq 0$ 
where $\gcd(n ,m)=1$ and $1 < n< m$, and the bounded derived category $\Db(\tails A)$ of $\tails A$.
Also, let $\nabla A$ be the Beilinson algebra of A and $\HH^{r}(\nabla A)$ the Hochschild cohomology groups of $\nabla A$.
Then \textcolor{red}{we have}
\[
\chi_{\HH}(\HH^{\bullet}(\nabla A))=
\begin{cases}
m+4 \quad &\text{if $n=2$},\\
n+m \quad &\text{if $n \geq 3$}, \\
\end{cases}
\quad \neq 2(n+m)
=\rank K_{0}(\Db(\tails A))\textcolor{red}{, } 
\]
where \textcolor{red}{we set}
\[
\chi_{\HH}(\HH^{\bullet}(\nabla A)):=\sum_{{k} \in \mathbb{Z}}(-1)^{k}\dim_{\bbk} \HH^{k}(\nabla A).
\]
\textcolor{red}{Hence} $\mathfrak{s}$ does not act unipotently on $K_{0}(\Db(\tails A))$.
\end{cor}
\begin{proof}
Assume that $\mathfrak{s}$ acts unipotently on $K_{0}(\Db(\tails A))$.
Let $M$ be the Gram matrix.
Then $\mathfrak{s}=M^{-1}M^{T}$ is unipotent (see \cite{BP}).
By Proposition \ref{prop-gram}, the Coxeter matrix $\Phi_{\nabla A}$ of $\nabla A$
 is defined by $-M^{-T}M$ (see Happel \cite[Section 1]{H2}). 
Then we have $-\tr \Phi_{\nabla A}=\tr M^{-T}M= \tr M^{-1}M^{T} = \tr \mathfrak{s}$.
Since all eigenvalues of a unipotent matrix are equal to $1$, 
$\tr \mathfrak{s}=2(n+m)$ holds.
By Happel's trace formula (\cite[Theorem 2.2]{H2}), 
\textcolor{blue}{ we have} $\chi_{\HH}(\HH^{\bullet}(\nabla A))= -\tr \Phi_{\nabla A}=2(n+m)$.
On the other hand, by Theorem \ref{HHdim-IM}, 
$\chi_{\HH}(\HH^{\bullet}(\nabla A))=
\begin{cases}
m+4 \quad &\text{if $n=2$,}\\
n+m \quad &\text{if $n \geq 3$.}\\
\end{cases}$
\textcolor{red}{But, this} is a contradiction. 
Therefore, $\mathfrak{s}$ does not act unipotently. 
\end{proof}
\begin{rem}\label{rem-cor}
By Corollary \ref{cor-exiPV}, we obtain the same conclusion \textcolor{red}{in} \cite[Section 3]{IU}; 
when $m > n > 1$, $\Db(\tails A)$ is not equivalent to the derived category of any smooth projective surface. 
This result is the facts corresponding to results by \cite{Bel} and \cite{IU}: 
\begin{itemize}
  \item \cite[Remark 26]{Bel}\textup{:} 
  In the case that $(n,m)=(1,1)$, the Serre functor unipotently acts on $\Db(\tails A)$.
  \item \cite[Proposition 3.2]{IU}\textup{:} 
  In the case that $(n,m)=(1,2)$, the Serre functor unipotently acts on $\Db(\tails A)$.  
  \textcolor{magenta}{On the other hand, when $(n,m)=(1,m)$ with $m \geq 3$, the Serre functor unipotently does not act on $\Db(\tails A)$.}
\end{itemize}
\end{rem}
\section{Basis of Hochschild cohomology groups}
\label{sec-basis}
\textcolor{magenta}{As in Section \ref{HHgroups}, let $A=A(\alpha, \beta)$ be a graded down-up algebra with weights $\mathrm{deg}\ x=n$, $\mathrm{deg}\ y=m$ and $\beta\neq0$, 
where $m \geq n >1$ and $\mathrm{gcd}(n, m)=1$. Let $\nabla A$ denote the Beilinson algebra of $A$.}
\textcolor{red}{In this section, we compute bases of the Hochschild cohomology groups of $\nabla A$ in order to determine (or study) their Yoneda ring structure.}
To compute \textcolor{red}{bases} of $\HH^{1}(\nabla A)$ and $\HH^{2}(\nabla A)$,
we define ${\bbk}$-vector subspaces of $\widehat{P^{1}}$ as 
\[
\widehat{P^{1}}_{L_{1}}:= 
\langle \tau_{x_1}, \dots , \tau_{x_{n+2m}}, \tau_{y_1}, \dots , \tau_{y_{2n+m}} \rangle_{\bbk},\,
\widehat{P^{1}}_{L_{2}}:=
\langle \tau_{y_{m+2}}^{x^m}, \dots , \tau_{y_1}^{x^m} \rangle_{\bbk},\,
\widehat{P^{1}}^{\star}_{L_{2}}:= 
\langle \tau_{x_1}^{y}, \tau_{x_2}^{y}, \tau_{x_3}^{y}
\rangle_{\bbk},
\]
and subspaces of $\widehat{P^{2}}$ as 
\begin{align*}
&
\widehat{P^{2}}_{L_{1}}:= 
\langle \tau_{f_{i_1}}^{yx^2}, \tau_{f_{i_2}}^{xyx}, 
\tau_{g_{j_1}}^{y^2 x}, \tau_{g_{j_2}}^{yxy}
\mid 1\leq i_{1}, i_{2} \leq m, 1\leq j_{1}, j_{2} \leq n \rangle_{\bbk}, \\
&
\widehat{P^{2}}_{L_{2}}:=
\langle \tau_{f_i}^{x^{m+2}}, \tau_{g_1}^{yx^{m+1}}, \tau_{g_1}^{xyx^m} \mid 1\leq i \leq m \rangle_{\bbk},\, 
\widehat{P^{2}}_{L_{2}}^{\star}:=
\langle \tau_{g_1}^{y^3}, \tau_{f_1}^{y^2x}, \tau_{f_1}^{yxy} \rangle_{\bbk}.
\end{align*}
\textcolor{red}
{We denote by
$\widehat{\partial^{2}}_{L_{1}}:=\widehat{\partial^{2}}|_{\widehat{P^{2}}_{L_{1}}}$, 
$\widehat{\partial^{2}}_{L_{2}}:=\widehat{\partial^{2}}|_{\widehat{P^{2}}_{L_{2}}}$, 
$\widehat{\partial^{2}}^{\star}_{L_{2}}:=\widehat{\partial^{2}}|
_{\widehat{P^{2}}^{\star}_{L_{2}}}$ 
the restrictions of $\widehat{\partial^{2}}$ to the indicated subspaces.}
\textcolor{magenta}{Let $\rho_{1}=\{ \tau_{x_1}, \dots , \tau_{x_{n+2m}},\tau_{y_1}, \dots , \tau_{y_{2n+m}} \}$ be an ordered basis of $\widehat{P^{1}}_{L_{1}}$, 
and let $\rho_{2}=
\{\tau_{f_1}^{yx^2},\, \ldots ,\,\tau_{f_m}^{yx^2},\, 
\tau_{g_1}^{y^2 x},\, \ldots,\, \tau_{g_n}^{y^2 x},\, 
\tau_{g_1}^{yxy},\, \ldots,\, \tau_{g_n}^{yxy},\, 
\tau_{f_1}^{xyx},\, \ldots,\, \tau_{f_m}^{xyx} \}
$ 
be the ordered basis of $\widehat{P^{2}}_{L_{1}}$ given above. 
Let $L_{1,1}$ be the matrix defined in Section \ref{HHgroups} 
corresponding to the restriction $\widehat{\partial^{2}}_{L_{1,1}}$. 
We denote the $\bbk$-linear map represented by $L_{1,1}$ with respect to the bases 
$\rho_{1}$ and $\rho_{2}$ 
by $\widehat{\partial^{2}}_{L_{1,1}}:\,\widehat{P^{1}}_{L_{1}}\rightarrow \widehat{P^{2}}$.}
\begin{lem}
We set $N:=\Ker \widehat{\partial^{2}}_{L_{1,1}}$.
The kernel of \textcolor{magenta}{$\widehat{\partial^{2}}$} 
\textcolor{red}{in \eqref{resol-1}} can be 
\textcolor{red}{described} as follows\textup{:}

{\large
\begin{enumerate}
\item If $m=n=1$, then 
\begin{align*}
\Ker \widehat{\partial^{2}} \hspace{-2pt} =\hspace{-3pt}
\begin{cases}
N \oplus \langle 
\tau_{x_{2}},\, 
\tau_{y_2}^{x},\, 
\beta \tau_{y_{3}}^{x} + \tau_{y_{1}}^{x},\, 
\tau_{x_2}^{y},\, \beta \tau_{x_{3}}^{y} + \tau_{x_{1}}^{y} 
\rangle_{\bbk} 
&\hspace{-8pt} \text{\textup{for Case I (and Case $1$),}}\\
N
\oplus \langle 
(\frac{\alpha}{2})^{2}\tau_{y_3}^{x}+ (\frac{\alpha}{2}) \tau_{y_{2}}^{x} + \tau_{y_{1}}^{x},\, 
(\frac{\alpha}{2})^{2}\tau_{x_3}^{y}+ (\frac{\alpha}{2}) \tau_{x_{2}}^{y} + \tau_{x_{1}}^{y} 
\rangle_{\bbk} 
&\hspace{-8pt} \text{\textup{for Case I\hspace{-1.2pt}I and Case $2$,}}\\
N
&\hspace{-8pt} \text{\textup{for Case I\hspace{-1.2pt}I and Case $3$.}}\\
\end{cases}
\end{align*}
\item If $m>n=1$, then 
\begin{align*}
\Ker \widehat{\partial^{2}} \hspace{-2pt} =\hspace{-3pt}
\begin{cases}
N \oplus \left\langle \Dsum_{r=1}^{m} (-1)^{r-1}\tau_{x_{n+r}},\, 
\Dsum_{r=2}^{m+2} \lambda_{r-1}\tau^{x^{m}}_{y_{r}},\, 
\beta \Dsum_{r=1}^{m+2} \lambda_{r-2}\tau^{x^{m}}_{y_{r}}
\right\rangle_{\bbk}
& \hspace{-8pt} \text{\textup{for Case I (and Case $1$),}}\\
N
\oplus \left\langle
\Dsum_{r=2}^{m+2} \lambda_{r-1}\tau^{x^{m}}_{y_{r}},\, 
\beta \Dsum_{r=1}^{m+2} \lambda_{r-2}\tau^{x^{m}}_{y_{r}}
\right\rangle_{\bbk}
& \hspace{-8pt} \text{\textup{for Case I\hspace{-1.2pt}I and Case $1$,}}\\
N
\oplus \left\langle 
\Dsum_{r=1}^{m+2} (\frac{\alpha}{2})^{r-1}\tau^{x^{m}}_{y_{r}}
\right\rangle_{\bbk}
& \hspace{-8pt} \text{\textup{for Case I\hspace{-1.2pt}I and Case $2$,}}\\
N
& \hspace{-8pt} \text{\textup{for Case I\hspace{-1.2pt}I and Case $3$.}}\\
\end{cases}
\end{align*}
\item If $m>n>1$, then 
$
\Ker \widehat{\partial^{2}}=
\begin{cases}
N \oplus \left\langle
\Dsum_{r=1}^{m} (-1)^{r-1}\tau_{x_{n+r}}
\right\rangle_{\bbk}
& \text{\textup{for Case I,}}\\ 
N
& \text{\textup{for Case I\hspace{-1.2pt}I.}} \\
\end{cases}
$
\end{enumerate}
}
\end{lem}

\begin{proof}
By Proposition \ref{par-IU} and Lemma \ref{par-IM}, it is enough to compute the kernel spaces of $L_{1}$ and $L_{2}$.
By the proof of Lemma \ref{rank-L1},
$\Dsum_{r=1}^{m} (-1)^{r-1}\tau_{x_{n+r}} \in \Ker \widehat{\partial^{2}}_{L_{1}}$
if and only if \textrm{Case I} holds.
By Lemma \ref{rank-L1}, 
\[
\Ker \widehat{\partial^{2}}_{L_{1}}=
\begin{cases}
N \oplus \left\langle\Dsum_{r=1}^{m} (-1)^{r-1}\tau_{x_{n+r}}\right\rangle_{\bbk}
& \text{\textup{for Case I,}}\\
N & \text{\textup{for Case I\hspace{-1.2pt}I.}}\\
\end{cases}
\]
The basis of kernel space $L_{2}$ is obtained by the proof of \cite[Lemma 2.8]{IU}.
\end{proof}
\begin{prop}
\label{basHH-1}
The ${\bbk}$-vector space $\HH^{1}(\nabla A)$ is \textcolor{red}{described as follows}. 
\textcolor{red}{A} basis of $\HH^{1}(\nabla A)$ \textcolor{red}{is given in} 
\textup{Table \ref{table-HH1}}\textup{:}

\renewcommand{\arraystretch}{2.5}
\begin{longtable}{|c|c|c|c||c|} 
\caption{List of basis of $\HH^{1}(\nabla A)$} \label{table-HH1}\\ \hline
$n$ & $m$ & \textup{Cond.$1$}& \textup{Cond.$2$} & {Basis of $\HH^{1}(\nabla A)$}
\\ \hline
\endfirsthead
$n=1$& $m=1$ & \textup{Case I} & \textup{Case $1$} &
\hspace{-8pt}
$\left\{
\left[\Dsum_{r=1}^{n+2m} \tau_{x_{r}}\right],  
\left[\tau_{x_{2}}\right],  
\left[\tau_{y_2}^{x}\right], 
\left[\beta \tau_{y_{3}}^{x} + \tau_{y_{1}}^{x}\right], 
\left[\tau_{x_2}^{y}\right], 
\left[\beta \tau_{x_{3}}^{y} + \tau_{x_{1}}^{y}\right]
\right\}$
\\[6pt] \hline
$n=1$ & $m=1$ & \textup{Case I\hspace{-1.2pt}I} & \textup{Case $2$} &
\hspace{-8pt}
$\left\{
\left[\Dsum_{r=1}^{n+2m} \tau_{x_{r}}\right],  
\left[(\frac{\alpha}{2})^{2}\tau_{y_3}^{x}+ (\frac{\alpha}{2}) \tau_{y_{2}}^{x} + \tau_{y_{1}}^{x}\right],   
\left[(\frac{\alpha}{2})^{2}\tau_{x_3}^{y}+ (\frac{\alpha}{2}) \tau_{x_{2}}^{y} + \tau_{x_{1}}^{y}\right] 
\right\} $
\\[6pt] \hline
$n=1$ & $m=1$ & \textup{Case I\hspace{-1.2pt}I} & \textup{Case $3$} &
\hspace{-8pt}
$\left\{
\left[\Dsum_{r=1}^{n+2m} \tau_{x_{r}}\right]
\right\}$
\\[6pt] \hline
$n=1$ & $m>n $ & \textup{Case I} & \textup{Case $1$} &
\hspace{-8pt}
$\left\{
\left[\Dsum_{r=1}^{n+2m} \tau_{x_{r}}\right], 
\left[\Dsum_{r=1}^{m} (-1)^{r-1}\tau_{x_{n+r}}\right],  
\left[\Dsum_{r=2}^{m+2} \lambda_{r-1}\tau^{x^{n}}_{y_{r}}\right],  
\left[\beta \Dsum_{r=1}^{m+2} \lambda_{r-2}\tau^{x^{n}}_{y_{r}}\right]
\right\}$
\\[6pt]\hline
$n=1$ & $m>n$ & \textup{Case I\hspace{-1.2pt}I} & \textup{Case $1$} &
\hspace{-8pt}
$\left\{
\left[\Dsum_{r=1}^{n+2m} \tau_{x_{r}}\right], 
\left[\Dsum_{r=2}^{m+2} \lambda_{r-1}\tau^{x^{m}}_{y_{r}}\right],  
\left[\beta \Dsum_{r=1}^{m+2} \lambda_{r-2}\tau^{x^{m}}_{y_{r}}\right]
\right\}$
\\[6pt] \hline
$n=1$ & $m>n$ & \textup{Case I\hspace{-1.2pt}I} & \textup{Case $2$} &
\hspace{-8pt}
$\left\{
\left[\Dsum_{r=1}^{n+2m} \tau_{x_{r}}\right], 
\left[\Dsum_{r=1}^{m+2} (\frac{\alpha}{2})^{r-1}\tau^{x^{m}}_{y_{r}}\right]
\right\}$
\\[6pt] \hline
$n=1$ & $m>n$ & \textup{Case I\hspace{-1.2pt}I} & \textup{Case $3$} &
\hspace{-8pt}
$\left\{
\left[\Dsum_{r=1}^{n+2m} \tau_{x_{r}}\right]
\right\}$
\\[6pt] \hline
$n>1$ & $m>n$ & \textup{Case I} & \hrulefill &
\hspace{-8pt}
$\left\{
\left[\Dsum_{r=1}^{n+2m} \tau_{x_{r}}\right], 
\left[\Dsum_{r=1}^{m} (-1)^{r-1}\tau_{x_{n+r}}\right]
\right\}$
\\[6pt] \hline
\textcolor{red}{$n>1$} & \textcolor{red}{$m>n$} & \textup{Case I\hspace{-1.2pt}I} & \hrulefill &
\hspace{-8pt}
$\left\{
\left[\Dsum_{r=1}^{n+2m} \tau_{x_{r}}\right]
\right\}$
\\[6pt] \hline
\end{longtable}
\renewcommand{\arraystretch}{2.5}
\end{prop}
\begin{proof}
\textcolor{magenta}{For $m \geq n \geq 1$, in \textup{Case I\hspace{-1.2pt}I and Case $3$,} 
we have $\Image \widehat{\partial^{1}} \subset \Ker \widehat{\partial^{2}}=N$. 
Since $\widehat{\partial^{1}}$ and $N$ do not depend on the parameters $\alpha, \beta$, it follows that $\Image \widehat{\partial^{1}} \subset N$ for all $\alpha$, $\beta$.}
By Lemma \ref{rank-L1}, \textcolor{red}{we have}
\begin{align*}
\dim_{\bbk} (N/\Image \widehat{\partial^{1}})
&=\dim_{\bbk} N -\dim_{\bbk} \Image \widehat{\partial^{1}}
=(3(n+m)-\rank L_{1,1})-(\dim_{\bbk}\widehat{P^{0}}-\dim_{\bbk} \Ker\widehat{\partial^{1}})\\
&=(3(n+m)-(n+m))-(2(n+m)-1)=1.
\end{align*}
To prove 
$\left\langle \left[ \Dsum_{r=1}^{n+2m} \tau_{x_{r}} \right] \right\rangle_{\bbk} 
= N / \Image \widehat{\partial^{1}}$ as $\bbk$-vector spaces, 
we 
show that $\Dsum_{r=1}^{n+2m} \tau_{x_{r}} \in N \setminus \Image \widehat{\partial^{1}}$.
It is clear that $\Dsum_{r=1}^{n+2m} \tau_{x_{r}} \in N$. 
\textcolor{red}{Using proof by contradiction, we 
show that $\Dsum_{r=1}^{n+2m} \tau_{x_{r}} \notin \Image \widehat{\partial^{1}}$.}
Assume that $\Dsum_{r=1}^{n+2m} \tau_{x_{r}} \in \Image\widehat{\partial^{1}}$. 
Then there exists $\phi \in \widehat{P^{0}}$ such that 
$\widehat{\partial^{1}} (\phi)=\Dsum_{r=1}^{n+2m} \tau_{x_{r}}\,\cdots (1)$. 
By Lemma \ref{bas-IM}, there exists $\mu_{e_{r}}\in {\bbk}$ such that 
$\phi = \Dsum_{r=1}^{n+2m} \mu_{e_{r}}\tau_{e_{r}} \in \widehat{P^{0}}\,\cdots (2)$. 
By (1) and (2), for $h \in \cG_{1}$,  \textcolor{red}{we have}
\begin{align*}
\widehat{\partial^{1}} (\phi)(s(h)\otimes t(h))
&=\phi \partial^{1}(s(h)\otimes t(h))
=\phi (s(h)\otimes s(h)h-ht(h)\otimes t(h))\\
&=\sum_{r=1}^{n+2m} \mu_{e_{r}}\tau_{e_{r}}(s(h)\otimes s(h)h-ht(h)\otimes t(h))
=\mu_{s(h)}h-\mu_{t(h)}h\,\,\cdots (3).
\end{align*}
On the other hand, by (1) and (2), for $h \in \cG_{1}$, 
\textcolor{red}{we have}
\begin{align*}
\widehat{\partial^{1}} (\phi)(s(h)\otimes t(h))&=
\left(\Dsum_{r=1}^{n+2m} \tau_{x_{r}}\right)(s(h)\otimes t(h))
=
\begin{cases}
x_{r} & \text{if $h=x_{r}$ for $1 \leq r \leq 2(n+m)$},\\
0 & \text{otherwise}
\end{cases}
\quad \cdots (4). 
\end{align*}
By (3) and (4),  
for $p \in \{ 1,\ldots ,n+2m \}$ and $q \in \{ 1,\ldots ,2n+m \}$, 
\textcolor{red}{the following equality hold: }
\begin{center}
$\mu_{e_{p+n}}-\mu_{e_{p}}=-1$ and $\mu_{e_{q+m}}-\mu_{e_{q}}=0 \,\cdots (5) $.
\end{center}
%
%
%
%
By (5), for all $k \in \{1,\ldots ,n\}$, 
\textcolor{red}{the following equality hold: }
\begin{align*}
\mu_{e_{n+2m+{k}}}-\mu_{e_{k}}
&=(\mu_{e_{n+2m+{k}}}-\mu_{e_{2m+{k}}})
+(\mu_{e_{2m+{k}}}-\mu_{e_{m+{k}}})
+(\mu_{e_{m+{k}}}-\mu_{e_{k}})=-1
\,\,\cdots (6).
\end{align*}
\textcolor{blue}{Let $a$ be the quotient and $b$ the remainder when $2(n+m)$ is divided by $n$,}
that is, we take integers $a,b \in \mathbb{Z}$ such that 
$2(n+m)=an+b$ and $0\leq b <n$ $\cdots (7)$. 
When $b=0$, the former case does not happen. 
In this case, we consider each part as zero. 
By (6) and (7), for all $q \in \{b+1, \ldots ,n\}$, 
we obtain \textcolor{red}{equalities}
\begin{align*}
\begin{cases}
\mu_{e_{2(n+m)-b+{p}}}-\mu_{e_{p}}
=\Dsum_{r=1}^{a} (\mu_{e_{{p}+nr}}-\mu_{e_{{p}+n(r-1)}})
=-a 
& \text{ for $1 \leq {p} \leq b$}\,\,\cdots (8);\\
\mu_{e_{n+2m-b+{q}}}-\mu_{e_{q}}
=\Dsum_{r=1}^{a-1} (\mu_{e_{{q}+nr}}-\mu_{e_{{q}+n(r-1)}})
=-(a-1)
& \text{ for $b+1 \leq {q} \leq n$}\,\,\cdots (9). 
\end{cases}
\end{align*}
By (8) and (9), 
we have
\begin{align*}
-n&=\Dsum_{r=1}^{n}(\mu_{e_{n+2m+r}}-\mu_{e_{r}})=\Dsum_{r=1}^{b}(\mu_{e_{2(n+m)-b+r}}-\mu_{e_{r}})
+\Dsum_{r=b+1}^{n}(\mu_{e_{n+2m-b+r}}-\mu_{e_{r}})\\
&= -ab-(n-b)(a-1)
=-n-2m.
\end{align*}
But, this is a contradiction. 
Therefore, $\Dsum_{r=1}^{n+2m} \tau_{x_{r}} \in N \setminus \Image \widehat{\partial^{1}}$, 
\textcolor{red}{this proves the claim.}
\end{proof}
Next, we compute a basis of $\HH^{2}(\nabla A)$. 
By Proposition \ref{par-IU} and Lemma \ref{par-IM}, 
\textcolor{red}{we have}
\begin{align*}
\Image \widehat{\partial^{2}} =
\begin{cases}
\Image \widehat{\partial^{2}}_{L_{1}} \oplus \Image \widehat{\partial^{2}}_{L_{2}} \oplus \Image \widehat{\partial^{2}}^{\star}_{L_{2}} 
& \text{if $1=n=m$,}\\
\Image \widehat{\partial^{2}}_{L_{1}} \oplus \Image \widehat{\partial^{2}}_{L_{2}} 
& \text{if $1=n < m$,} \\
\Image \widehat{\partial^{2}}_{L_{1}} 
& \text{if $1 < n < m$.} \\
\end{cases}
\end{align*}
\begin{prop}\label{bas-P2L1}
\begin{enumerate}
\item
\textcolor{red}{If} \textup{Case I} holds, 
\textcolor{red}{then} $\Image \widehat{\partial^{2}}_{L_{1}}$ has a ${\bbk}$-basis

$
\{
\tau^{y x^2}_{f_p} + \tau^{y^2 x}_{g_p},\,
\tau^{y x^2}_{f_q} + \tau^{y x^2}_{f_n+q},\,
\tau^{y x^2}_{f_m-n+r} + \tau^{y^2 x}_{g_r}
\mid 1 \leq p \leq n,\, 1 \leq q \leq m-n,\, 1 \leq r \leq n-1
\}.
$
\item
\textcolor{red}{If} \textup{Case I\hspace{-1.2pt}I} holds, 
\textcolor{red}{then} $\Image \widehat{\partial^{2}}_{L_{1}}$ has a ${\bbk}$-basis
\begin{align*}
\left\{\begin{array}{c|c}
\beta \tau^{y x^2}_{f_p} + \beta \tau^{y^2 x}_{g_p} + \alpha\tau^{yxy}_{f_p},\,
\beta \tau^{y x^2}_{f_q} + \beta \tau^{y x^2}_{f_{n+q}} + \alpha\tau^{yxy}_{f_{n+q}},&
1 \leq p \leq n,\, 1 \leq q \leq m-n,\,\\
\beta \tau^{y x^2}_{f_{m-n+r}} + \beta \tau^{y^2 x}_{g_r} + \alpha\tau^{yxy}_{f_{m-n+r}}&
1 \leq r \leq n
\end{array}\right\}.
\end{align*}
\end{enumerate}
Also, we obtain
$
\widehat{P^{2}}_{L_{1}}=
\begin{cases}
\Image \widehat{\partial^{2}}_{L_{1}}\oplus
\langle
\tau^{xyx}_{f_i}, \tau^{yxy}_{g_j}, \tau^{y^2 x}_{g_n}
| 1 \leq i \leq n, 1 \leq j \leq m
\rangle_{\bbk}
& \text{\textup{for Case I,}}\\
\Image \widehat{\partial^{2}}_{L_{1}}\oplus
\langle
\tau^{xyx}_{f_i}, \tau^{yxy}_{g_j}
| 1 \leq i \leq n, 1 \leq j \leq m
\rangle_{\bbk}
& \text{\textup{for Case I\hspace{-1.2pt}I.}}\\
\end{cases}
$
\end{prop}
\begin{proof}
\textcolor{red}{We prove the statement only for \textup{Case I}; 
\textup{Case I\hspace{-1.2pt}I} is analogous.}
\textcolor{red}{The ${\bbk}$-linear endomorphism $\phi$ of $\widehat{P^{2}}$ is defined by}
\begin{align*}
\begin{cases}
\tau_{f_p}^{yx^2}
\mapsto
\tau^{y x^2}_{f_p} + \tau^{y^2 x}_{g_p}
&\text{ for } 1 \leq p \leq n,
\\
\tau_{f_q}^{yx^2}
\mapsto
\tau^{y x^2}_{f_{q}} + \tau^{y x^2}_{f_{n+q}}
&\text{ for } 1 \leq q \leq m-n,
\\
\tau_{g_r}^{y^2 x}
\mapsto
\tau^{y x^2}_{f_{m-n+r}} + \tau^{y^2 x}_{g_{r}}
&\text{ for } 1 \leq r \leq n-1,
\end{cases}
\quad
\begin{cases}
\tau^{y^2 x}_{g_n} 
\mapsto
\tau^{y^2 x}_{g_n}, 
&
\\
\tau^{xyx}_{f_i}
\mapsto
\tau^{xyx}_{f_i}
&\text{ for } 1 \leq i \leq m,
\\
\tau^{yxy}_{g_j}
\mapsto
\tau^{yxy}_{g_j}
&\text{ for } 1 \leq j \leq n.
\\
\end{cases}
\end{align*}
Then the matrix representation $M_{\phi}$ of $\phi$ 
with respect to the ordered basis
\begin{center}
$\{ \tau_{f_1}^{yx^2}, \ldots ,\tau_{f_m}^{yx^2}, 
\tau_{g_1}^{y^2 x},\ldots ,\tau_{g_n}^{y^2 x},$ 
$\tau_{g_1}^{yxy},\ldots ,\tau_{g_n}^{yxy},
\tau_{f_1}^{xyx}, \ldots ,\tau_{f_m}^{xyx} \} $
\,\,\,is\,\,\,
$\begin{pNiceArray}{c:c}
\Block{1-1}<\large>{C_{n+m-1}} & \Block{1-1}<\large>{$0$} \\
\hdottedline
\Block{1-1}<\large>{$0$}& \Block{1-1}<\large>{E_{n+m+1}} \\
\end{pNiceArray},$
\end{center}
\textcolor{red}{where $C$ is the circulant matrix with associated polynomial $f(x)=x^{m}+1$, 
$E_{n+m+1}$ is the $(n+m+1)$-th identity matrix 
and $C_{n+m-1}$ is the principal submatrix of $C$ deleting the last row and column.}
\textcolor{magenta}{Hence,}
\textcolor{red}{by Lemma \ref{rank-C} \eqref{rank-C-equ2},}
\textcolor{magenta}{$\rank M_{\phi} = 2(n+m)$ and then $\phi$ is a ${\bbk}$-linear automorphism of $\widehat{P^{2}}$.}
By Lemma \ref{rank-L1}, 
\textcolor{red}{this proves the claim}. 
\end{proof}
\begin{prop}\label{bas-P2L2}
We set
$V:=\{
\tau_{f_{r}}^{x^{m+2}}-\lambda_{m+2-r}\tau_{g_{1}}^{yx^{m+1}}+\lambda_{m+1-r}\tau_{g_{1}}^{xyx^{m}}
\mid 1\leq r \leq m
\}$.
\begin{enumerate}
  \item \textcolor{red}{If} \textup{Case $1$} holds, 
   \textcolor{red}{then} $\widehat{P^{2}}_{L_{2}}$ has a ${\bbk}$-basis $V$.
  \item \textcolor{red}{If} \textup{\textcolor{magenta}{Case $2$}} holds, 
  \textcolor{red}{then} $\widehat{P^{2}}_{L_{2}}$ has a ${\bbk}$-basis $V\cup \{-\alpha\tau^{yx^{m+1}}_{g_{1}}+2\tau^{xyx^m}_{g_{1}}\}$.
  \item \textcolor{red}{If} \textup{\textcolor{magenta}{Case $3$}} holds, 
  \textcolor{red}{then} $\widehat{P^{2}}_{L_{2}}$ has a ${\bbk}$-basis $V \cup \{\tau^{yx^{m+1}}_{g_{1}}, \tau^{xyx^m}_{g_{1}}\}$.
\end{enumerate}
Also, we obtain 
$
\widehat{P^{2}}_{L_{2}}=
\begin{cases}
\Image \widehat{\partial^{2}}_{L_{2}}\oplus
\langle \tau^{yx^{m+1}}_{g_{1}}, \tau^{xyx^m}_{g_{1}}\rangle_{\bbk}
&\text{\textup{for Case $1$,}}\\
\Image \widehat{\partial^{2}}_{L_{2}}\oplus
\langle \tau^{xyx^m}_{g_{1}}\rangle_{\bbk}
&\text{\textup{for Case $2$,}}\\
\Image \widehat{\partial^{2}}_{L_{2}}
&\text{\textup{for Case $3$.}}\\
\end{cases}
$
\end{prop}
\begin{proof}
\textcolor{red}{We have} 
{\small
\begin{align*}
(L_{2})^{T} &=
\left(\begin{array}{ccccccc}
1 &&&&& -\lambda_{2} & \lambda_{1} \\
-\alpha& 1 &&&& -\beta \lambda_{1} & -\beta \lambda_{0} \\
-\beta & -\alpha&&&&& \\
& -\beta & \ddots &&&& \\
&&& 1 &&& \\
&&& -\alpha& 1 && \\
&&& -\beta & -\alpha& \beta \lambda_{m} & \lambda_{m+1} \\
&&&& -\beta & -\beta \lambda_{m+1} & -\lambda_{m+2} \\
\end{array}\right)
\rightarrow 
\left(\begin{array}{ccccccc}
1 &&&&& -\lambda_{2} & \lambda_{1} \\
& 1 &&&& -\lambda_{3} & \lambda_{2} \\
-\beta & -\alpha&&&&& \\
& -\beta & \ddots &&&& \\
&&& 1 &&& \\
&&& -\alpha& 1 && \\
&&& -\beta & -\alpha& \beta \lambda_{m} & \lambda_{m+1} \\
&&&& -\beta & -\beta \lambda_{m+1} & -\lambda_{m+2} \\
\end{array}\right)\\
&\rightarrow
\dots
\rightarrow 
\left(\begin{array}{ccccccc}
1 &&&&& -\lambda_{2} & \lambda_{1} \\
& 1 &&&& -\lambda_{3} & \lambda_{2} \\
&& \ddots &&& \vdots & \vdots \\
&&& 1 && -\lambda_{m} & \lambda_{m-1} \\
&&&& 1 &  -\lambda_{m+1} & \lambda_{m} \\
&&& -\beta & -\alpha& \beta \lambda_{m} & \lambda_{m+1} \\
&&&& -\beta & -\beta \lambda_{m+1} & -\lambda_{m+2} \\
\end{array}\right)
\rightarrow 
\left(\begin{array}{ccccccc}
1 &&&&& -\lambda_{2} & \lambda_{1} \\
& 1 &&&& -\lambda_{3} & \lambda_{2} \\
&& \ddots &&& \vdots & \vdots \\
&&& 1 && -\lambda_{m} & \lambda_{m-1} \\
&&&& 1 &  -\lambda_{m+1} & \lambda_{m} \\
&&&&& -\alpha\lambda_{m+1} & 2\lambda_{m+1} \\
&&&&& -2\beta \lambda_{m+1} & -\alpha\lambda_{m+1} \\
\end{array}\right).
\end{align*}
}
Therefore, by the proof of \cite[Lemma 2.8]{IU}, Lemma \ref{par-IM} and Proposition \ref{par-IU}, \textcolor{red}{the claim follows}.  
\end{proof}
\begin{prop}\label{basHH-2}
The ${\bbk}$-vector space $\HH^{2}(\nabla A)$ is  is \textcolor{red}{described as follows}. 
\textcolor{red}{A} basis of $\HH^{2}(\nabla A)$ \textcolor{red}{is given in} \textup{Table \ref{table-HH2}}\textup{:} 
\renewcommand{\arraystretch}{1.7}
\begin{longtable}{|c|c|c|c||c|}
\caption{List of basis of $\HH^{2}(\nabla A)$} \label{table-HH2}\\ \hline
$n$ & $m$ & \textup{Cond.$1$}& \textup{Cond.$2$} & {Basis of $\HH^{2}(\nabla A)$} \\ \hline
\endfirsthead
$n=1$ & $m=1$ & \textup{Case I } & \textup{Case $1$} & 
$\left\{
\Big[\tau_{g_1}^{y^2x}\Big], \Big[\tau_{g_1}^{yxy}\Big], \Big[\tau_{f_1}^{xyx}\Big],
 \Big[\tau_{g_1}^{yx^2}\Big], \Big[\tau_{g_1}^{xyx}\Big], \Big[\tau_{f_1}^{y^2x}\Big],
 \Big[\tau_{f_1}^{yxy}\Big], \Big[\tau_{g_1}^{x^3}\Big], \Big[\tau_{f_1}^{y^3}\Big]
\right\}$ 
\\
&&&&\\[-20pt]
\hline
$n=1$ & $m=1$ & \textup{Case I\hspace{-1.2pt}I } & \textup{Case $2$} & 
$\left\{
\Big[\tau_{g_1}^{yxy}\Big], \Big[\tau_{f_1}^{xyx}\Big], 
\Big[\tau_{g_1}^{xyx}\Big],
\Big[\tau_{f_1}^{yxy}\Big], 
\Big[\tau_{g_1}^{x^3}\Big], \Big[\tau_{f_1}^{y^3}\Big]
\right\}$
\\[2pt]\hline
$n=1$ & $m=1$ & \textup{Case I\hspace{-1.2pt}I } & \textup{Case $3$} & 
$\left\{
\Big[\tau_{g_1}^{yxy}\Big], \Big[\tau_{f_1}^{xyx}\Big], 
\Big[\tau_{g_1}^{x^3}\Big], \Big[\tau_{f_1}^{y^3}\Big]
\right\}$
\\[2pt]\hline
$n=1$ & $m=2$ & \textup{Case I\hspace{-1.2pt}I} & \textup{Case $1$}&
$\left\{
\Big[\tau_{f_1}^{xyx}\Big], \Big[\tau_{f_2}^{xyx}\Big], \Big[\tau_{g_1}^{yxy}\Big], 
\Big[\tau_{g_1}^{yx^3}\Big], \Big[\tau_{g_1}^{xyx^2}\Big], 
\Big[\tau_{g_1}^{x^5}\Big], \Big[\tau_{f_1}^{y^2}\Big], \Big[\tau_{f_2}^{y^2}\Big]
\right\}$
\\[2pt]\hline
$n=1$ & $m=2$ & \textup{Case I\hspace{-1.2pt}I} & \textup{Case $2$}&
$\left\{
\Big[\tau_{f_1}^{xyx}\Big], \Big[\tau_{f_2}^{xyx}\Big], \Big[\tau_{g_1}^{yxy}\Big], 
\Big[\tau_{g_1}^{xyx^2}\Big], 
\Big[\tau_{g_1}^{x^5}\Big], \Big[\tau_{f_1}^{y^2}\Big], \Big[\tau_{f_2}^{y^2}\Big]
\right\}$
\\[2pt]\hline
$n=1$ & $m=2$ & \textup{Case I\hspace{-1.2pt}I} & \textup{Case $3$}&
$\left\{
\Big[\tau_{f_1}^{xyx}\Big], \Big[\tau_{f_2}^{xyx}\Big], \Big[\tau_{g_1}^{yxy}\Big], 
\Big[\tau_{g_1}^{x^5}\Big], \Big[\tau_{f_1}^{y^2}\Big], \Big[\tau_{f_2}^{y^2}\Big]
\right\}$
\\[2pt]\hline
$n=1$ & $m \geq 3$ & \textup{Case I} & \textup{Case $1$}&
$\left\{
\Big[\tau^{xyx}_{f_i}\Big], \Big[\tau^{yxy}_{g_1}\Big], \Big[\tau^{y^2 x}_{g_1}\Big], 
\Big[\tau^{yx^{m+1}}_{g_{1}}\Big], \Big[\tau^{xyx^m}_{g_{1}}\Big], 
\Big[\tau^{x^{2m+1}}_{g_{1}}\Big]
\middle|\, 1 \leq i \leq m \right\}$
\\[2pt]\hline
$n=1$ & $m \geq 3$ & \textup{Case I\hspace{-1.2pt}I} & \textup{Case $1$}&
$\left\{
\Big[\tau^{xyx}_{f_i}\Big], \Big[\tau^{yxy}_{g_1}\Big], 
\Big[\tau^{yx^{m+1}}_{g_{1}}\Big], \Big[\tau^{xyx^m}_{g_{1}}\Big], 
\Big[\tau^{x^{2m+1}}_{g_{1}}\Big]
\middle|\, 1 \leq i \leq m \right\}$
\\[2pt]\hline
$n=1$ & $m \geq 3$ & \textup{Case I\hspace{-1.2pt}I} & \textup{Case $2$}&
$\left\{
\Big[\tau^{xyx}_{f_i}\Big], \Big[\tau^{yxy}_{g_1}\Big], 
\Big[\tau^{xyx^m}_{g_{1}}\Big], 
\Big[\tau^{x^{2m+1}}_{g_{1}}\Big]
\middle|\, 1 \leq i \leq m \right\}$
\\[2pt]\hline
$n=1$ & $m \geq 3$ & \textup{Case I\hspace{-1.2pt}I} & \textup{Case $3$}&
$\left\{
\Big[\tau^{xyx}_{f_i}\Big], \Big[\tau^{yxy}_{g_1}\Big], 
\Big[\tau^{x^{2m+1}}_{g_{1}}\Big]
\middle|\, 1 \leq i \leq m \right\}$
\\[2pt]\hline
$n=2$ & $m > n$ & \textup{Case I} & \hrulefill &
$\left\{
\Big[\tau^{xyx}_{f_i}\Big], \Big[\tau^{yxy}_{g_j}\Big], \Big[\tau^{y^2 x}_{g_n}\Big], 
\Big[\tau_{g_{1}}^{x^{1+m}}\Big], \Big[\tau_{g_{2}}^{x^{1+m}}\Big]
\middle|\, 1 \leq i \leq m, 1 \leq j \leq n
\right\}$
\\[2pt]\hline
$n=2$ & $m > n$ & \textup{Case I\hspace{-1.2pt}I} & \hrulefill &
$\left\{
\Big[\tau^{xyx}_{f_i}\Big], \Big[\tau^{yxy}_{g_j}\Big], 
\Big[\tau_{g_{1}}^{x^{1+m}}\Big], \Big[\tau_{g_{2}}^{x^{1+m}}\Big]
\middle|\, 1 \leq i \leq m, 1 \leq j \leq n
\right\}$
\\[2pt]\hline
$n \geq 3$ & $m > n$ & \textup{Case I} & \hrulefill &
$\left\{
\Big[\tau^{xyx}_{f_i}\Big], \Big[\tau^{yxy}_{g_j}\Big], \Big[\tau^{y^2 x}_{g_n}\Big]
\middle|\, 1 \leq i \leq m, 1 \leq j \leq n
\right\}$ 
\\[2pt]
\hline
$n \geq 3$ & $m > n$ & \textup{Case I\hspace{-1.2pt}I} & \hrulefill &
$\left\{
\Big[\tau^{xyx}_{f_i}\Big], \Big[\tau^{yxy}_{g_j}\Big]
\middle|\, 1 \leq i \leq m, 1 \leq j \leq n
\right\}$
\\[2pt]
\hline
\end{longtable}
\renewcommand{\arraystretch}{1.0}
\end{prop}

\begin{proof}
By Lemmas \ref{bas-IU}, \ref{bas-IM}, Propositions \ref{bas-P2L1} \textcolor{red}{and} \ref{bas-P2L2}, 
\textcolor{red}{hence the claim follows.}
\end{proof}
\section{Hochschild cohomology rings}
\label{sec-5}
\textcolor{red}{In this section, we compute the ring structure of $\bigoplus_{r \geq 0} \HH^{r}(\nabla A)$ with respect to the Yoneda product,}
\textcolor{blue}{using the notation introduced in Section $\ref{sec-basis}$.}
First, we define $(\nabla A)^{\mathrm{e}}$-homomorphisms $\sigma^{p}_{0}: P^{1} \rightarrow P^{0}$ and $\sigma^{q}_{1}: P^{2} \rightarrow P^{1}$ for $p, q \in \{1, 2, 3, 4, 5, 3', 4', 5'\}$. 
\begin{enumerate}
\item
\textcolor{red}{Assume $m \geq n \geq 1$. 
Let
$h^{1}:=\Dsum_{r=1}^{n+2m} \tau_{x_{r}}$. 
For $a \in \cG_{1}$, define }
\begin{align*}
&\sigma^{1}_{0}(s(a)\otimes t(a)):=
\begin{cases}
(s(e_{i})\otimes t(e_{i}))x_{i} &\text{if}\, a=x_{i} \; \text{for}\, 1 \leq i \leq n+2m ,\\
0 & \text{otherwise}\textcolor{red}{.}\\
\end{cases}
\end{align*}
\textcolor{red}{For }$a \in \cG_{2}$, \textcolor{red}{define}
\begin{align*}
&\sigma^{1}_{1}(s(a)\otimes t(a)):=
\begin{cases}
(s(x_{i})\otimes t(x_{i}))x_{i+n}y_{i+2n}
-\alpha(s(x_{i})\otimes t(x_{i}))y_{i+n}x_{i+n+m}\\
\quad
-\alpha x_{i}(s(y_{i+n})\otimes t(y_{i+n}))x_{i+n+m}
-2\beta(s(y_{i})\otimes t(y_{i}))x_{i+m}x_{i+n+m}\\
\quad
-\beta y_{i}(s(x_{i+m})\otimes t(x_{i+m}))x_{i+n+m}
 &\hspace{-100pt}\text{if}\, a=f_{i}\; \text{for}\, 1 \leq i \leq m ,\\
-\alpha(s(y_{j})\otimes t(y_{j}))x_{j+m}y_{j+n+m}
-\beta(s(y_{j})\otimes t(y_{j}))y_{j+m}x_{j+2m}\\
\quad
-\beta y_{j}(s(y_{j+m})\otimes t(y_{j+m}))x_{j+2m}
 &\hspace{-100pt}\text{if}\, a=g_{j} \; \text{for}\, 1 \leq i \leq n.\\
\end{cases}
\end{align*}
Moreover, 
\textcolor{red}{assume that \textup{Case I} holds. 
Let 
$h^{2}:=\Dsum_{r=1}^{m} (-1)^{r-1}\tau_{x_{n+r}}$. 
For $a \in \cG_{1}$, }
\text{\textcolor{red}{define}}\\
\begin{align*}
&\sigma^{2}_{0}(s(a)\otimes t(a)):=
\begin{cases}
(-1)^{i-n}(s(e_{i})\otimes t(e_{i}))x_{i} &\text{if}\, a=x_{i}\; \text{for}\, n+1 \leq i \leq n+m ,\\
0 & \text{otherwise}\textcolor{red}{.}\\
\end{cases}
\end{align*}
\textcolor{red}{For }$a \in \cG_{2}$, \textcolor{red}{define}
\begin{align*}
&\sigma^{2}_{1}(s(a)\otimes t(a)):=
\begin{cases}
(-1)^{i}((s(x_{i})\otimes t(x_{i}))x_{i+n}y_{i+2n}\\
\quad
-\beta(s(y_{i})\otimes t(y_{i}))x_{i+m}x_{i+n+m})
 &\text{if}\, a=f_{i}\; \text{for}\, 1 \leq i \leq n ,\\
(-1)^{i}(s(x_{i})\otimes t(x_{i}))x_{i+n}y_{i+2n}
 &\text{if}\, a=f_{i}\; \text{for}\, n+1 \leq i \leq m ,\\
0  &\text{if}\, a=g_{j} \; \text{for}\, 1 \leq i \leq n.\\
\end{cases}
\end{align*}
\item
\begin{enumerate}
\item
\textcolor{red}{Assume that $m \geq n =1$ and Case $1$ holds. 
Let $h^{3}:=\Dsum_{r=2}^{m+2} \lambda_{r-1}\tau^{x^{m}}_{y_{r}}$, 
$h^{4}:=\beta \Dsum_{r=1}^{m+2} \lambda_{r-2}\tau^{x^{m}}_{y_{r}}$. 
For $a \in \cG_{1}$, define}
\begin{align*}
&\sigma^{3}_{0}(s(a)\otimes t(a)):=
\begin{cases}
\lambda_{j-1}(s(e_{j})\otimes t(e_{j}))x_{j}x_{j+1}\cdots x_{j+m-1} &\text{if}\, a=y_{j} \; \text{for}\, 1 \leq j \leq 2n+m,\\
0 & \text{otherwise};\\
\end{cases}
\end{align*}
\begin{align*}
&\sigma^{4}_{0}(s(a)\otimes t(a)):=
\begin{cases}
\beta\lambda_{j-2}(s(e_{j})\otimes t(e_{j}))x_{j}x_{j+1}\cdots x_{j+m-1} &\text{if}\, a=y_{j} \; \text{for}\, 1 \leq j \leq 2n+m ,\\
0 & \text{otherwise}\textcolor{red}{.}\\
\end{cases}
\end{align*}
\textcolor{red}{For }$a \in \cG_{2}$, \textcolor{red}{define}
\begin{align*}
&\sigma^{3}_{1}(s(a)\otimes t(a)):=
\begin{cases}
\beta \lambda_{i-1} (s(x_{i})\otimes t(x_{i}))x_{i+1} x_{i+2} \cdots x_{i+m+1}\\
\quad
+\lambda_{i+1} x_{i}(s(x_{i+1})\otimes t(x_{i+1})) x_{i+2} \cdots x_{i+m+1}
 &\text{if}\, a=f_{i}\, \text{for}\, 1 \leq i \leq m ,\\
(s(x_{1})\otimes t(x_{1}))x_{2} \cdots x_{m+1} y_{m+2}\\
\quad
-\beta \lambda_{m} (s(y_{1})\otimes t(y_{1}))x_{m+1} \cdots x_{2m} x_{2m+1}
 & \text{if}\, a=g_{1};\\
\end{cases}
\end{align*}
\begin{align*}
&\sigma^{4}_{1}(s(a)\otimes t(a)):=
\begin{cases}
\beta^{2} \lambda_{i-2} (s(x_{i})\otimes t(x_{i}))x_{i+1} x_{i+2} \cdots x_{i+m+1}\\
\quad
+\beta \lambda_{i} x_{i}(s(x_{i+1})\otimes t(x_{i+1})) x_{i+2} \cdots x_{i+m+1}
 &\text{if}\, a=f_{i}\; \text{for}\, 1 \leq i \leq m ,\\
-\alpha\beta\lambda_{m}y_{1}(s(x_{m+1})\otimes t(x_{m+1})) x_{m+2} \cdots x_{2m+1}\\
\quad
+\beta\lambda_{m}(s(x_{1})\otimes t(x_{1}))y_{2} x_{m+2} \cdots x_{2m+1}\\
\quad
+\beta\lambda_{m}x_{1}(s(y_{2})\otimes t(y_{2})) x_{m+2} \cdots x_{2m+1}
 & \text{if}\, a=g_{1}.\\
\end{cases}
\end{align*}
\item 
\textcolor{red}{Assume that $m \geq n=1$ and Case $2$ holds.  
Let $h^{5}:=\Dsum_{r=1}^{m+2} 
\left(\frac{\alpha}{2}\right)^{r-1}\tau^{x^{m}}_{y_{r}}$. 
For $a \in \cG_{1}$, define}
\begin{align*}
&\sigma^{5}_{0}(s(a)\otimes t(a)):=
\begin{cases}
(\frac{\alpha}{2})^{j-1}(s(e_{j})\otimes t(e_{j}))x_{j}x_{j+1}\cdots x_{j+m-1} &\text{if}\, a=y_{j} \; \text{for}\, 1 \leq j \leq 2n+m ,\\
0 & \text{otherwise}\textcolor{red}{.}\\
\end{cases}
\end{align*}
\textcolor{red}{For }$a \in \cG_{2}$, \textcolor{red}{define}
\begin{align*}
&\sigma^{5}_{1}(s(a)\otimes t(a)):=
\begin{cases}
-(\frac{\alpha}{2})^{i+1} (s(x_{i})\otimes t(x_{i}))x_{i+1} x_{i+2} \cdots x_{i+m+1}\\
\quad
+(\frac{\alpha}{2})^{i+1}x_{i}(s(x_{i+1})\otimes t(x_{i+1})) x_{i+2} \cdots x_{i+m+1}
 &\text{if}\, a=f_{i}\; \text{for}\, 1 \leq i \leq m ,\\
(\frac{\alpha}{2})(s(x_{1})\otimes t(x_{1}))x_{2} \cdots x_{m+1} y_{m+2}\\
\quad
+(\frac{\alpha}{2})^{m+1}(s(x_{1})\otimes t(x_{1}))y_{2} x_{m+2} \cdots x_{2m+1}\\
\quad
-(\frac{\alpha}{2})^{m+2}(s(y_{1})\otimes t(y_{1}))x_{m+1} \cdots x_{2m+1}\\
\quad
+(\frac{\alpha}{2})^{m+1}x_{1}(s(y_{2})\otimes t(y_{2}))x_{m+2} \cdots x_{2m+1}\\
\quad
-2(\frac{\alpha}{2})^{m+2}y_{1}(s(x_{m+1})\otimes t(x_{m+1}))x_{m+2} \cdots x_{2m+1}
 & \text{if}\, a=g_{1}.\\
\end{cases}
\end{align*}
\end{enumerate}
\item
\begin{enumerate}
\item
\textcolor{red}{Assume that $m=n=1$ and Case $1$ holds. 
Let $h^{3'}:=\tau_{x_2}^{y}$, 
$h^{4'}:=\beta \tau_{x_{3}}^{y} + \tau_{x_{1}}^{y}$. 
For $a \in \cG_{1}$, define}
\begin{align*}
&\sigma^{3'}_{0}(s(a)\otimes t(a)):=
\begin{cases}
(s(e_{2}) \otimes t(e_{2}))y_{2} & \text{if}\, a=x_{2},\\
0 & \text{otherwise};\\
\end{cases}
\end{align*}
\begin{align*}
&\sigma^{4'}_{0}(s(a)\otimes t(a)):=
\begin{cases}
\beta(s(e_{i}) \otimes t(e_{i}))y_{1} & \text{if}\, a=x_{1},\\
(s(e_{i}) \otimes t(e_{i}))y_{3} & \text{if}\, a=x_{3},\\
0 & \text{otherwise}\textcolor{red}{.}\\
\end{cases}
\end{align*}
\textcolor{red}{For }$a \in \cG_{2}$, \textcolor{red}{define}
\begin{align*}
&\sigma^{3'}_{1}(s(a)\otimes t(a)):=
\begin{cases}
(s(x_{1})\otimes t(x_{1}))y_{2}y_{3}
-\beta(s(y_{1})\otimes t(y_{1}))y_{2}x_{3} &\text{if}\, a=f_{1},\\
0 & \text{if}\, a=g_{1};\\
\end{cases}
\end{align*}
\begin{align*}
&\sigma^{4'}_{1}(s(a)\otimes t(a)):=
\begin{cases}
-\beta(y_{1}(s(x_{2})\otimes t(x_{2}))y_{3}
+(s(y_{1})\otimes t(y_{1}))x_{2}y_{3}) &\text{if}\, a=f_{1},\\
-\beta(y_{1}(s(y_{2})\otimes t(y_{2}))y_{3}
+(s(y_{1})\otimes t(y_{1}))y_{2}y_{3}) & \text{if}\, a=g_{1}.\\
\end{cases}\\
\end{align*}
\item
\textcolor{red}{Assume that $m=n=1$, and Case $2$ holds. 
Let $h^{5'}:= \tau_{x_{3}}^{y} + \left(\frac{\alpha}{2}\right) \tau_{x_{2}}^{y} + \left(\frac{\alpha}{2}\right)^{2} \tau_{x_{1}}^{y}$. 
For $a \in \cG_{1}$, define}
\begin{align*}
&\sigma^{5'}_{0}(s(a)\otimes t(a)):=
\begin{cases}
(\frac{\alpha}{2})^{3-j}(s(e_{j})\otimes t(e_{j}))y_{j} &\text{if}\, a=x_{j} \; \text{for}\, 1 \leq j \leq 3 ,\\
0 & \text{otherwise}\textcolor{red}{.}\\
\end{cases}
\end{align*}
\textcolor{red}{For }$a \in \cG_{2}$, \textcolor{red}{define}
\begin{align*}
&\sigma^{5'}_{1}(s(a)\otimes t(a)):=
\begin{cases}
\left(\frac{\alpha}{2}\right)^{2}y_{1}(s(x_{2})\otimes t(x_{2}))y_{3}
+\left(\frac{\alpha}{2}\right)^{2}(s(y_{1})\otimes t(y_{1}))x_{2}y_{3}\\
\quad
+\left(\frac{\alpha}{2}\right)^{3}(s(y_{1})\otimes t(y_{1}))y_{2}x_{3}
-\left(\frac{\alpha}{2}\right)(s(x_{1})\otimes t(x_{1}))y_{2}y_{3}\\
\quad
-2\left(\frac{\alpha}{2}\right)x_{1}(s(y_{2})\otimes t(y_{2}))y_{3}
&\text{if}\, a=f_{1},\\
-\left(\frac{\alpha}{2}\right)^{2}(s(y_{1})\otimes t(y_{1}))y_{2}y_{3}
+\left(\frac{\alpha}{2}\right)^{2}y_{1}(s(y_{2})\otimes t(y_{2}))y_{3} & \text{if}\, a=g_{1}.\\
\end{cases}\\
\end{align*}
\end{enumerate}
\end{enumerate}
\textcolor{red}{We also} recall a useful lemma as follows\textup{:}

\begin{lem}[{\cite[Lemma 2.3]{IU}}]\label{IU-Lem2.3}
\textcolor{red}{In a down-up algebra $A=A(\alpha, \beta)$ with $\beta \neq 0$, 
the following equality holds for all $i \geq 1$\rm{:} 
$
x^{i}y=\beta\lambda_{i-1} yx^{i}+\lambda_{i}xyx^{i-1}
$.}
\end{lem}


To compute the Yoneda product on $\bigoplus_{r \geq 0}\HH^{r}(\nabla A)$, 
we make sure that the above maps form a chain map, respectively. 


\begin{prop} \label{ring-1}
Let $A=A(\alpha, \beta)$ be a down-up algebra 
with weights $\deg x = n$, $\deg y = m$ and $\beta \neq 0$ where $\gcd(n, m)=1$ and $m \geq n \geq 1$, 
and $\nabla A$ \textcolor{red}{its} Beilinson algebra.
Then the homomorphisms defined as above give the following commutative diagrams for $q \in 
\{1,\,2,\,3,\,4,\,5,\,3',\,4',\,5'\}$\textup{;} 

\hspace{120pt}
\xymatrix{
\cdots \ar[r] & P^{2} \ar[r]^{\partial^{2}} \ar[d]^{\sigma^{q}_{1}} & P^{1} \ar[r]^{h^{q}} \ar[d]^{\sigma^{q}_{0}} & \nabla A \ar[r] \ar[d]^{id_{\nabla A}} & 0 \\
\cdots \ar[r] & P^{1} \ar[r]^{(-1)\partial^{1}} & P_{0} \ar[r]^{\partial^{0}} & \nabla A \ar[r] & 0.\\}
\end{prop}

\begin{proof}
\textcolor{red}{We prove only the case $q = 1$; the other cases are analogous.}

For $a \in \cG_{1}$, \textcolor{red}{we have}
\begin{align*}
\partial^{0}\sigma^{1}_{0}(s(a)\otimes t(a))&=
\begin{cases}
\partial^{0}(s(e_{i})\otimes t(e_{i}))x_{i} &\text{if}\, a=x_{i} \, \text{for}\, 1 \leq i \leq n+2m ,\\
0 & \text{otherwise};\\
\end{cases}\\
&=
\begin{cases}
x_{i} &\text{if}\, a=x_{i} \, \text{for}\, 1 \leq i \leq n+2m ,\\
0 & \text{otherwise.}\\
\end{cases}
=h^{1}(s(a) \otimes t(a)).
\end{align*}
Therefore, $\partial^{0}\sigma^{1}_{0}=h^{1}$ \textcolor{red}{holds}.

For $a \in \cG_{2}$, \textcolor{red}{we have}
\begin{align*}
(-1)\partial^{1}\sigma^{1}_{1}(s(a) \otimes t(a))
&=
\begin{cases}
(-1)\partial^{1}(
(s(x_{i})\otimes t(x_{i}))x_{i+n}y_{i+2n}
-\alpha(s(x_{i})\otimes t(x_{i}))y_{i+n}x_{i+n+m}\\
\quad
-\alpha x_{i}(s(y_{i+n})\otimes t(y_{i+n}))x_{i+n+m}
-2\beta(s(y_{i})\otimes t(y_{i}))x_{i+m}x_{i+n+m}\\
\quad
-\beta y_{i}(s(x_{i+m})\otimes t(x_{i+m}))x_{i+n+m})
 &\hspace{-120pt}\text{if}\, a=f_{i}\, \text{ for}\, 1 \leq i \leq m ,\\
(-1)\partial^{1}(
-\alpha(s(y_{j})\otimes t(y_{j}))x_{j+m}y_{j+n+m}
-\beta(s(y_{j})\otimes t(y_{j}))y_{j+m}x_{j+2m}\\
\quad
-\beta y_{j}(s(y_{j+m})\otimes t(y_{j+m}))x_{j+2m})
 &\hspace{-120pt} \text{if}\, a=g_{j} \, \text{ for}\, 1 \leq i \leq n;\\
\end{cases}\\
&=
\begin{cases}
x_{i}(s(e_{i+n})\otimes t(e_{i+n}))x_{i+n}y_{i+2n}
-\alpha x_{i}y_{i+n}(s(e_{i+n+m})\otimes t(e_{i+n+m}))x_{i+n+m}\\
\quad
-\beta y_{i}(s(e_{i+m})\otimes t(e_{i+m}))x_{i+m}x_{i+n+m}
+\beta (s(e_{i})\otimes t(e_{i}))y_{i}x_{i+m}x_{i+n+m}\\
\quad
-\beta y_{i}x_{i+m}(s(e_{i+n+m})\otimes t(e_{i+n+m}))x_{i+n+m}&\hspace{-120pt}\text{if}\, a=f_{i}\, \text{ for}\, 1 \leq i \leq m ,\\
(s(e_{j})\otimes t(e_{j}))x_{j}y_{j+n}y_{j+n+m}
-\alpha y_{j}(s(e_{j+m})\otimes t(e_{j+m}))x_{j+m}y_{j+n+m}\\
\quad
-\beta y_{j}y_{j+m}(s(e_{j+2m})\otimes t(e_{j+2m}))x_{j+2m}
 &\hspace{-120pt} \text{if}\, a=g_{j} \, \text{ for}\, 1 \leq i \leq n.\\
\end{cases}\\
\end{align*}

On the other hand, \textcolor{red}{we have}
\begin{align*} 
\sigma^{1}_{0}\partial^{2}(s(a) \otimes t(a))
&=
\begin{cases}
&\sigma^{1}_{0}((
(s(x_{i}) \otimes t(x_{i}))x_{i+n}y_{i+2n}
+x_{i}(s(x_{i+n}) \otimes t(x_{i+n}))y_{i+2n}\\
&\quad\quad
+x_{i}x_{i+n}(s(y_{i+2n}) \otimes t(y_{i+2n}))
)\\
&\quad
-\alpha(
(s(x_{i}) \otimes t(x_{i}))y_{i+n}x_{i+n+m}
+x_{i}(s(y_{i+n}) \otimes t(y_{i+n}))x_{i+n+m}\\
&\quad\quad
+x_{i}y_{i+n}(s(x_{i+n+m}) \otimes t(x_{i+n+m}))
)\\
&\quad
-\beta(
(s(y_{i}) \otimes t(y_{i}))x_{i+m}x_{i+n+m}
+y_{i}(s(x_{i+m}) \otimes t(x_{i+m}))x_{i+n+m}\\
&\quad\quad
+y_{i}x_{i+m}(s(x_{i+n+m}) \otimes t(x_{i+n+m}))
))
\hspace{55pt}\text{if $h=f_{i}$,}\\
&
\sigma^{1}_{0}((
(s(x_{j}) \otimes t(x_{j}))y_{j+n}y_{j+n+m}
+x_{j}(s(y_{j+n}) \otimes t(y_{j+n}))y_{j+n+m}\\
&\quad\quad
+x_{j}y_{j+n}(s(y_{j+n+m}) \otimes t(y_{j+n+m}))
)\\
&\quad
-\alpha(
(s(y_{j}) \otimes t(y_{j}))x_{j+m}y_{j+n+m}
+y_{j}(s(x_{j+m}) \otimes t(x_{j+m}))y_{j+n+m}\\
&\quad\quad
+y_{j}x_{j+m}(s(y_{j+n+m}) \otimes t(y_{j+n+m}))
)\\
&\quad
-\beta(
(s(y_{j}) \otimes t(y_{j}))y_{j+m}x_{j+2m}
+y_{j}(s(y_{j+m}) \otimes t(y_{j+m}))x_{j+2m}\\
&\quad\quad
+y_{j}y_{j+m}(s(x_{j+2m}) \otimes t(x_{j+2m}))
))
\hspace{70pt}\text{if $h=g_{j}$,}\\
\end{cases}\\
&=
\begin{cases}
&x_{i}(s(e_{i+n}) \otimes t(e_{i+n}))x_{i+n}y_{i+2n}
-\alpha x_{i}y_{i+n}(s(e_{i+n+m}) \otimes t(e_{i+n+m}))x_{i+n+m}\\
&\quad
+\beta(s(e_{i}) \otimes t(e_{i}))y_{i}x_{i+m}x_{i+n+m}
-\beta y_{i}x_{i+m}(s(e_{i+n+m}) \otimes t(e_{i+n+m}))x_{i+n+m}\\
&\quad
-\beta y_{i}(s(e_{i+m}) \otimes t(e_{i+m}))x_{i+m}x_{i+n+m}
\hspace{75pt}\text{if $h=f_{i}$,}\\
&
((s(e_{j}) \otimes t(e_{j}))x_{j}y_{j+n}y_{j+n+m}
-\alpha y_{j}(s(e_{j+m}) \otimes t(e_{j+m}))x_{j+m}y_{j+n+m}\\
&\quad
-\beta y_{j}y_{j+m}(s(e_{j+2m}) \otimes t(e_{j+2m}))x_{j+2m})
\hspace{67pt}\text{if $h=g_{j}$.}\\
\end{cases}\\
\end{align*}
Therefore, $\sigma^{1}_{0}\partial^{2}=(-1)\partial^{1}\sigma^{1}_{1}$ \textcolor{red}{holds}.
\end{proof}


\begin{prop}\label{cup-prod}
We use \textcolor{red}{the notations} in Proposition {\rm \ref{ring-1}}. 
We obtain the following \textcolor{red}{equality} in $\HH^{2}(\nabla A)$\textup{:}
\begin{enumerate}
\item When $m \geq n \geq 1$ and \textup{Case I} holds\textup{,} 
$
\left[h^{1} \sigma^{2}_{1}\right] =m\left[\tau_{g_{n}}^{y^{2}x}\right].
$
\item When $m \geq n= 1$\textup{:}
\begin{itemize}
\item
\textup{For Case $1$, }
$
\left[h^{1}\sigma^{3}_{1}\right]=m\beta\lambda_{m}\left[\tau_{g_{1}}^{yx^{m+1}}\right], \,
\left[h^{1}\sigma^{4}_{1}\right]=-\beta\lambda_{m}\left[\tau_{g_{1}}^{xyx^{m}}\right], \,
\left[h^{3}\sigma^{4}_{1}\right]=\beta\lambda_{m}\left[\tau_{g_{1}}^{x^{2m+1}}\right].\,
$
\text{Moreover, for \textup{Case I, }}
$\left[h^{2}\sigma^{3}_{1}\right]=0$, 
$\left[h^{2}\sigma^{4}_{1}\right]=-\beta\left[\tau_{g_{1}}^{xyx^{m}}\right]$.
\item
\textup{For Case $2$, }
$\left[h^{1}\sigma^{5}_{1}\right]=0$. 
\end{itemize}
\item When $m=n=1$\textup{:}
\begin{itemize}
\item
\textup{For Case $1$,}
$
\left[h^{1}\sigma^{3'}_{1}\right]=\beta\left[\tau_{f_{1}}^{y^{2}x}\right],\;
\left[h^{1}\sigma^{4'}_{1}\right]=\left[h^{2}\sigma^{4'}_{1}\right]=-\beta\left[\tau_{f_{1}}^{yxy}\right],\;
\left[h^{3}\sigma^{4'}_{1}\right]=-\beta\left[\tau_{g_{1}}^{yxy}\right],\;
$

$
\left[h^{4}\sigma^{3'}_{1}\right]=-\beta\left[\tau_{f_{1}}^{xyx}\right],\;
\left[h^{3'}\sigma^{4'}_{1}\right]=-\beta\left[\tau_{f_{1}}^{y^{3}}\right],\;
\left[h^{2}\sigma^{3'}_{1}\right]=\left[h^{4}\sigma^{4'}_{1}\right]=\left[h^{3}\sigma^{3'}_{1}\right]=0.
$
\item 
\textup{For Case $2$, }
$
\left[h^{1}\sigma^{5'}\right]=
\left[h^{5}\sigma^{5'}\right]=0.
$
\end{itemize}
\end{enumerate}
\end{prop}

 
\begin{proof}
\textcolor{blue}{We prove only the assertions for} $h^{1} \sigma^{2}_{1}$ and $h^{1} \sigma^{3}_{1}$; \textcolor{red}{the remaining cases are similar.}

\textcolor{red}{If} Case I holds, 
\textcolor{red}{then,} for $a \in \cG_{2}$, \textcolor{red}{we have}
\vspace{-35pt}
\begin{align*}
h^{1} \sigma^{2}_{1}(s(a)\otimes t(a)) &=
\begin{cases}
(-1)^{i}x_{i}x_{i+n}y_{i+2n}
 &\text{if}\, a=f_{i}\, \text{ for }\, 1 \leq i \leq m ,\\
0  & \text{if}\, a=g_{j} \, \text{ for }\, 1 \leq i \leq n;\\
\end{cases}\\
&=
\begin{cases}
(-1)^{i}\beta y_{i}x_{i+m}x_{i+n+m}
 &\text{if}\, a=f_{i}\, \text{ for }\, 1 \leq i \leq m ,\\
0  & \text{if}\, a=g_{j} \, \text{ for }\, 1 \leq i \leq n.
\end{cases}
\end{align*}
\textcolor{blue}{We now define sequences $\{a_{r} \}_{r \geq 1}$ and
$\{b_{r} \}_{r \geq 1}$ by
\vspace{-30pt}
\[
b_{1}:=n,\, 
b_{r+1} = n a_{r} + b_{r} -(m-n), \,
a_{r}:=\left\lfloor \frac{m-b_{r}}{n}\right\rfloor, 
\vspace{-35pt}
\]
where $\lfloor \cdot \rfloor$ denotes the floor function.}
Clearly, the sequences satisfy
$m-n =n\left(\frac{m-b_{r}}{n}-1\right) + b_{r} < n a_{r} + b_{r} \leq n \left(\frac{m-b_{r}}{n}\right) + b_{r}= m$
 and $0< b_{r} \leq n$
  for $1 \leq r \leq n$.
Since $b_{r+1} -b_{r} = n a_{r} -(m-n)$,
we have
$b_{q}-b_{p}= n \Dsum_{r=p}^{q-1} a_{r} -(q-p)(m-n)$
for $q \geq p \geq 1$.
Since $ \gcd(m, n)=1$, the remainder of $b_{j}$ and $b_{i}$ by $n$ are not equal.
Then $\{1,2, \textcolor{red}{\ldots} , n\}=\textcolor{magenta}{\{b_{1}, b_{2}, \ldots ,b_{n}\}}$ and we have $\{1, 2, \textcolor{red}{\ldots} ,m\}=\coprod_{1 \leq p \leq n}\{nq+b_{p}\mid \, 0\leq q \leq a_{p} \}$.
\textcolor{red}{By} Hermite's identity (see \cite{M}), 
\textcolor{red}{we have}
\vspace{-15pt}
\[
\sum_{r=1}^{n}a_{r}
=\sum_{r=1}^{n}\left\lfloor \frac{m-n}{n} +\frac{n-b_{r}}{n}\right\rfloor
=\sum_{r=0}^{n-1}\left\lfloor \frac{m-n}{n} +\frac{r}{n} \right\rfloor
=\left\lfloor n \frac{m-n}{n} \right\rfloor
=m-n
\text{ and then  }
b_{n+1}=n.
\]
Hence, \textcolor{red}{we have}
\begin{align*}
\left[h^{1} \sigma^{2}_{1}\right]
&=\sum_{r=1}^{m}(-1)^{r}\left[\tau_{f_{r}}^{yx^2}\right]
\\
&=\sum_{p=1}^{n}\sum_{q=0}^{a_{p}}(-1)^{nq+b_{p}} \left[\tau_{f_{nq+b_{p}}}^{yx^2}\right]
&&\hspace{-120pt}
\textcolor{magenta}{\big(\because \{1, \ldots ,m\}=\coprod_{1 \leq p \leq n}\{nq+b_{p} \mid 0\leq q \leq a_{p} \}\big)}
\\
&=\sum_{p=2}^{n}\sum_{q=0}^{a_{p}}(-1)^{nq+b_{p}} \left[\tau_{f_{nq+b_{p}}}^{yx^2}\right]
+\sum_{q=0}^{a_{1}-1}(-1)^{nq+b_{1}} (q+1)\left[\tau_{f_{nq+b_{1}}}^{yx^2}+\tau_{f_{n(q+1)+b_{1}}}^{yx^2}\right]
\\
&\quad +(-1)^{na_{1}+b_{1}} (a_{1}+1)\left[\tau_{f_{na_{1}+b_{1}}}^{yx^2}\right]
\\
&=\sum_{p=2}^{n}\sum_{q=0}^{a_{p}}(-1)^{n\textcolor{red}{q}+b_{p}} \left[\tau_{f_{nq+b_{p}}}^{yx^2}\right]\\
&\quad +(-1)^{na_{1}+b_{1}} (a_{1}+1)\left(\left[\tau_{f_{na_{1}+b_{1}}}^{yx^2}+\tau_{g_{na_{1}+b_{1}-(m-n)}}^{y^{2}x}\right]-\left[\tau_{g_{b_{2}}}^{y^{2}x}+\tau_{f_{b_{2}}}^{yx^2}\right]\right)\\
&\quad +(-1)^{b_{2}} (a_{1}+1)\left[\tau_{f_{b_{2}}}^{yx^2}\right]\\
&=\sum_{p=2}^{n}\sum_{q=0}^{a_{p}}(-1)^{nq+b_{p}} \left[\tau_{f_{nq+b_{p}}}^{yx^2}\right]
+(-1)^{b_{2}} (a_{1}+1)\left[\tau_{f_{b_{2}}}^{yx^2}\right]
&&\hspace{-100pt}
\textcolor{magenta}{(\because \text{Proposition \ref{bas-IU} and Lemma \ref{par-IM}})}\\
&=(-1)^{b_{n+1}} \left(n+\sum_{r=1}^{n}a_{r}\right)\left[\tau_{f_{b_{n+1}}}^{yx^2}\right]\\
&=m\left[\tau_{g_{n}}^{y^{2}x}\right]+
(-1)^{n}m\left[\tau_{g_{n}}^{y^{2}x}+\tau_{f_{n}}^{yx^2}\right]
&&\hspace{-100pt}
\textcolor{magenta}{(\because \text{Hermite's identity})}\\
&=m\left[\tau_{g_{n}}^{y^{2}x}\right].
&&\hspace{-100pt}
\textcolor{magenta}{(\because \text{Proposition \ref{bas-IU} and Lemma \ref{par-IM}})}
\end{align*}
Therefore, we obtain the \textcolor{red}{equality} of $h^{1}\sigma_{1}^{2}$.

\textcolor{red}{If} $m \geq n=1$ and Case $1$ holds, 
\textcolor{red}{then, } for $a \in \cG_{2}$, \textcolor{red}{we have}
\vspace{-15pt}
\begin{align*}
h^{1} \sigma^{3}_{1}(s(a)\otimes t(a)) &=
\begin{cases}
\beta \lambda_{i-1} x_{i}x_{i+1} x_{i+2} \cdots x_{i+m+1}\\
\quad +\lambda_{i+1} x_{i}x_{i+1}x_{i+2} \cdots x_{i+m+1}
 &\text{if }\, a=f_{i}\, \text{ for }\, 1 \leq i \leq m ,\\
x_{1}x_{2} \cdots x_{m+1} y_{m+2}
 & \text{if }\, a=g_{1};\\
\end{cases}\\
&=
\begin{cases}
(\beta \lambda_{i-1}+\lambda_{i+1}) x_{i}x_{i+1} x_{i+2} \cdots x_{i+m+1}
 &\text{if }\, a=f_{i}\, \text{for}\, 1 \leq i \leq m ,\\
\beta\lambda_{m}y_{1}x_{m+1} \cdots x_{2m} x_{2m+1}
 & \text{if }\, a=g_{1}.
\end{cases}
\end{align*}
By \cite[Lemma 2.6 (2)]{IU} and $\lambda_{m+1}=0$, we have $\alpha^{2}+4\beta \neq 0$.
Let $u, v$ be the solutions of $t^2-\alpha t -\beta=0$ with $u \neq v$.
Since $\lambda_{0}=0$ and $\lambda_{1}\textcolor{red}{=0}$, we can write $\lambda_{r}=\dfrac{u^{r}-v^{r}}{u-v}$ for all $r \geq -1$.
Also, we have $0=\lambda_{m+1}=\dfrac{u^{m+1}-v^{m+1}}{u-v}$. 
\textcolor{red}{Then} there exists an ($m+1$)-th root of unity $\zeta \neq 1$ such that $\zeta = \dfrac{u}{v}$.
\textcolor{blue}{We have the following equalities\textup{:}}
\begin{align*}
\left[h^{1}\sigma^{3}_{1}\right]
&=\beta\lambda_{m}\left[\tau_{g_{1}}^{yx^{m+1}}\right]+\sum_{r=1}^{m}(\lambda_{r+1}+\beta\lambda_{r-1})\left[\tau_{f_{r}}^{x^{m+2}}\right]\\
&=\left(\beta\lambda_{m}+\sum_{r=1}^{m} \lambda_{m+2-r}(\lambda_{r+1}+\beta\lambda_{r-1})\right)\left[\tau_{g_{1}}^{yx^{m+1}}\right]
 -\left(\sum_{r=1}^{m} \lambda_{m+1-r}(\lambda_{r+1}+\beta\lambda_{r-1})\right)\left[\tau_{g_{1}}^{xyx^{m}}\right]\\
&\quad+\sum_{r=1}^{m}(\lambda_{r+1}+\beta\lambda_{r-1})\left[\tau_{f_{r}}^{x^{m+2}}-\lambda_{m+2-r}\tau_{g_{1}}^{yx^{m+1}}+\lambda_{m+1-r}\tau_{g_{1}}^{xyx^{m}}\right]\\
&=\frac{v^{m+3}}{(u-v)^{2}}\left(-\zeta(\zeta-1)(\zeta^{m}-1)+\sum_{r=1}^{m} (\zeta^{m+2-r}-1)((\zeta^{r+1}-1)-\zeta(\zeta^{r-1}-1))\right)\left[\tau_{g_{1}}^{yx^{m+1}}\right]\\
&\quad -\frac{v^{m+2}}{(u-v)^{2}}\left(\sum_{r=1}^{m} (\zeta^{m+1-r}-1)((\zeta^{r+1}-1)-\zeta(\zeta^{r-1}-1))\right)\left[\tau_{g_{1}}^{xyx^{m}}\right]\\
&=\frac{v^{m+3}}{(u-v)^{2}}\left((\zeta-1)^{2}+\sum_{r=1}^{m} ((\zeta-1)^{2}+(\zeta^{m+3-r}-\zeta^{m+2-r}+\zeta^{r}-\zeta^{r+1}))\right)\left[\tau_{g_{1}}^{yx^{m+1}}\right]\\
&\quad -\frac{v^{m+2}}{(u-v)^{2}}\left(\sum_{r=1}^{m} ((\zeta^{m+2-r}-\zeta^{r+1})+(\zeta^{r}-\zeta^{m+1-r}))\right)\left[\tau_{g_{1}}^{xyx^{m}}\right]\\
&=\frac{v^{m+3}}{(u-v)^{2}}((\zeta-1)^{2}+m(\zeta-1)^{2}-\zeta^{2}+2\zeta-1)\left[\tau_{g_{1}}^{yx^{m+1}}\right]
=m\beta\lambda_{m}\left[\tau_{g_{1}}^{yx^{m+1}}\right].
\end{align*}
\textcolor{blue}{Therefore, we obtain the equality for}
$\left[h^{1}\sigma^{3}_{1}\right]$.
\end{proof}

To describe the ring structures of the Hochschild cohomology groups $\HH^{r}(\nabla A)$, 
we set \textcolor{red}{the} notation as follows: 
For integers $a \geq 1$ and $b \geq 0$, the exterior algebra is denoted by 
\[
\Lambda(a, b):=
\begin{cases}
  \Lambda(\langle s_{p} \mid 1\leq p \leq a \rangle_{\bbk})
  & \text{if $b=0$}, \\
  \Lambda(\langle s_{p}, t_{q}\mid 1\leq p \leq a, 1 \leq q \leq b \rangle_{\bbk})
  & \text{if $b > 0$}, 
\end{cases}
\quad\text{where $\deg s_{p}=1$ and $\deg{t_{q}}=2.$}
\]

\begin{thm}
\label{thm-ring}
\textcolor{red}{Let  $A=A(\alpha, \beta)$ be a graded down-up algebra 
with $\deg x = n$, $\deg y = m$, $\beta \neq 0$, 
$\gcd(n, m)=1$ and $m \geq n \geq 1$, 
and let $\nabla A$ be its Beilinson algebra. 
Then, for each case listed below, there is an isomorphism
\[
\bigoplus_{{r} \geq 0}\HH^{r}(\nabla A) \cong \Lambda(a,b) \left/ \left(I + \bigoplus_{{p} \geq 3}(\Lambda(a,b))_{p}\right)\right.,
\]
where $a \geq 1$ and $b \geq 0$ are integers, and 
$I$ is a two-sided homogeneous ideal of the exterior algebra $\Lambda(a, b)$
specified in \textup{Table \ref{table-ring}}.}

\renewcommand{\arraystretch}{1.7}
\begin{longtable}{|c|c|c|c||c|c|}
\caption{\textcolor{red}{Number of generators $(a,b)$ and generators of the ideal $I$}}
\label{table-ring}\\ \hline
$n$ & $m$ & \textup{Cond.$1$}& \textup{Cond.$2$} & 
$(a, b)$ & The set of generators of $I$
\\ \hline
\endfirsthead
$n=1$ & $m=1$ & \textup{Case I} & \textup{Case $1$} & $(6,0)$ 
& 
$\{{s_{2}s_{3}, s_{2}s_{4}, s_{3}s_{4}, s_{5}s_{6}, s_{1}s_{5}-s_{2}s_{5},  s_{1}s_{6}-s_{2}s_{6}}\}$
\\ \hline
$n=1$ & $m=1$ & \textup{Case I\hspace{-1.4pt}I} & \textup{Case $2$} & $(3,6)$ &
$\{s_{r}s_{r'}\mid 1\leq r, r' \leq 3\}$
\\ \hline
$n=1$ & $m=1$ & \textup{Case I\hspace{-1.4pt}I} & \textup{Case $3$} & $(1,4)$ &
$\{0\}$
\\ \hline
$n=1$ & $m=2$ & \textup{Case I\hspace{-1.4pt}I} & \textup{Case $1$} & $(3,5)$ &
$\{0\}$
\\ \hline
$n=1$ & $m=2$ & \textup{Case I\hspace{-1.4pt}I} & \textup{Case $2$} & $(2,7)$ &
$\{0\}$
\\ \hline
$n=1$ & $m=2$ & \textup{Case I\hspace{-1.4pt}I} & \textup{Case $3$} & $(1,6)$ &
$\{0\}$
\\ \hline
$n=1$ & \textcolor{red}{$m \geq 3$} & \textup{Case I} & \textup{Case $1$} & $(4,m+1)$ &
$\{s_{1}s_{4}-s_{2}s_{4}, s_{2}s_{3}\}$
\\ \hline
$n=1$ & \textcolor{red}{$m \geq 3$} & \textup{Case I\hspace{-1.4pt}I} & \textup{Case $1$} & $(3,m+1)$ &
$\{0\}$
\\ \hline
$n=1$ & $m \geq 3$ & \textup{Case I\hspace{-1.4pt}I} & \textup{Case $2$} & $(2,m+3)$ &
$\{0\}$
\\ \hline
$n=1$ & $m \geq 3$ & \textup{Case I\hspace{-1.4pt}I} & \textup{Case $3$} & $(1,m+2)$ &
$\{0\}$
\\ \hline
$n=2$ & $m > n$ & \textup{Case I} & \hrulefill & $(2,m+4)$ &
$\{0\}$
\\ \hline
$n=2$ & $m > n$ & \textup{Case I\hspace{-1.2pt}I} & \hrulefill & $(1,m+4)$ &
$\{0\}$
\\ \hline
$n \geq 3$ & $m > n$ & \textup{Case I} & \hrulefill & $(2,m+n)$ &
$\{0\}$
\\ \hline
$n \geq 3$ & $m > n$ & \textup{Case I\hspace{-1.2pt}I} & \hrulefill & $(1,m+n)$ &
$\{0\}$
\\ \hline
\end{longtable}
\renewcommand{\arraystretch}{1.0}
\end{thm}

\begin{proof}
\textcolor{red}{We use the symbols $(h_1, h_2, \ldots)$ defined at the beginning of Section \ref{sec-5}.}
\textcolor{red}{We prove for $n=1$, $m = 1$, \textup{Case $3$}, and \textup{Case I\hspace{-1.2pt}I} and $n=1$, $m \geq 3$, \textup{Case $1$}, and \textup{Case I}.
In other cases as well, it can be proven using the same method.}
\begin{enumerate}
\item \textcolor{red}{If} $n=1$, $m = 1$, \textup{Case $3$}, and \textup{Case I\hspace{-1.2pt}I} hold, 
\textcolor{red}{then} Theorem \ref{thm-Bel}, Propositions \ref{basHH-1} \textcolor{red}{and} \ref{basHH-2} provide
\begin{align*}
&\HH^{1}(\nabla A)
=
\left\langle
\left[\Dsum_{r=1}^{n+2m} \tau_{x_{r}}\right]
\right\rangle_{\bbk}, \\
&\HH^{2}(\nabla A)
=
\left\langle
\Big[\tau_{g_1}^{yxy}\Big], \Big[\tau_{f_1}^{xyx}\Big], 
\Big[\tau_{g_1}^{x^3}\Big], \Big[\tau_{f_1}^{y^3}\Big] \right\rangle_{\bbk},\\
&\HH^{i}(\nabla A)=0 \,\text{ for }\, i \geq 3.
\end{align*}
Moreover, by Theorem \ref{gr-com}, we have
$\HH^{1}(\nabla A)\smile \HH^{1}(\nabla A)=0.$
\textcolor{red}{Hence a ${\bbk}$-algebra morphism $\phi: \Lambda(1,4) \rightarrow \bigoplus_{{r} \geq 0}\HH^{r}(\nabla A)$ can be defined by}
\[
s_{1}
\mapsto [h_{1}];\quad
t_{1}
\mapsto \Big[\tau_{g_1}^{yxy}\Big];\quad
t_{2}
\mapsto \Big[\tau_{f_1}^{xyx}\Big];\quad
t_{3}
\mapsto \Big[\tau_{g_1}^{x^3}\Big];\quad
t_{4} 
\mapsto \Big[\tau_{f_1}^{y^3}\Big].
\]
Then $\phi$ is surjective $k$-algebra morphism and 
$\Ker \phi = \bigoplus_{{p} \geq 3}(\Lambda(a,b))_{p}$.
Therefore, \textcolor{red}{this proves the claim.} 
\item \textcolor{red}{If} $n=1$, $m \geq 3$, \textup{Case $1$}, and \textup{Case I}. hold, 
\textcolor{red}{then} 
Theorem \ref{thm-IU}, Propositions \ref{basHH-1} \textcolor{red}{and} \ref{basHH-2} provide
\begin{align*}
\HH^{1}(\nabla A)
&=
\left\langle
\left[\Dsum_{r=1}^{n+2m} \tau_{x_{r}}\right], 
\left[\Dsum_{r=1}^{m} (-1)^{r-1}\tau_{x_{n+r}}\right],  
\left[\Dsum_{r=2}^{m+2} \lambda_{r-1}\tau^{x^{n}}_{y_{r}}\right],  
\left[\beta \Dsum_{r=1}^{m+2} \lambda_{r-2}\tau^{x^{n}}_{y_{r}}\right]
\right\rangle_{\bbk}, \\
\HH^{2}(\nabla A)
&=
\left\langle
\Big[\tau^{xyx}_{f_i}\Big], \Big[\tau^{yxy}_{g_1}\Big], \Big[\tau^{y^2 x}_{g_1}\Big], 
\Big[\tau^{yx^{m+1}}_{g_{1}}\Big], \Big[\tau^{xyx^m}_{g_{1}}\Big], 
\Big[\tau^{x^{2m+1}}_{g_{1}}\Big]
\middle|\, 1 \leq i \leq m \right\rangle_{\bbk},\\
\HH^{i}(\nabla A)&=0 \,\text{ for }\, i \geq 3.
\end{align*}
Moreover, by Theorem \ref{gr-com} and Proposition \ref{cup-prod}, 
we summarize the 
\textcolor{red}{Yoneda} products on elements of $\HH^{1}(\nabla A)$ as the following table\textup{:}
\renewcommand{\arraystretch}{2.0}
\begin{longtable}{|c||c|c|c|c|}\hline
$\smile$ & $[h_{1}]$ & $[h_{2}]$ & $[h_{3}]$ & $[h_{4}]$ \\ \hline \hline
\endfirsthead
  \hline
  \endlastfoot
$[h_{1}]$ & $0$ & $m\left[\tau_{g_{1}}^{y^{2}x}\right]$ & $m\beta\lambda_{m}\left[\tau_{g_{1}}^{yx^{m+1}}\right]$ & $-\beta\lambda_{m}\left[\tau_{g_{1}}^{xyx^{m}}\right]$
\\ \hline
$[h_{2}]$ & $-m\left[\tau_{g_{1}}^{y^{2}x}\right]$ & $0$ & $0$ & $-\beta\left[\tau_{g_{1}}^{xyx^{m}}\right]$
\\ \hline
$[h_{3}]$ & $-m\beta\lambda_{m}\left[\tau_{g_{1}}^{yx^{m+1}}\right]$ & $0$ & $0$ & $\beta\lambda_{m}\left[\tau_{g_{1}}^{x^{2m+1}}\right]$
\\ \hline
$[h_{4}]$ & $\beta\lambda_{m}\left[\tau_{g_{1}}^{xyx^{m}}\right]$ & $\beta\left[\tau_{g_{1}}^{xyx^{m}}\right]$ & $-\beta\lambda_{m}\left[\tau_{g_{1}}^{x^{2m+1}}\right]$ & $0$ 
\\ 
\end{longtable}
\renewcommand{\arraystretch}{1.0}
\noindent
\textcolor{red}{Therefore,} we have
\[
\HH^{1}(\nabla A)\smile \HH^{1}(\nabla A)
=\left\langle 
\Big[\tau^{y^2 x}_{g_1}\Big], 
\Big[\tau^{yx^{m+1}}_{g_{1}}\Big], \Big[\tau^{xyx^m}_{g_{1}}\Big], 
\Big[\tau^{x^{2m+1}}_{g_{1}}\Big]  
\right\rangle_{\bbk}. 
\]
\textcolor{red}{Hence a ${\bbk}$-algebra morphism $\phi': \Lambda(4,m+1) \rightarrow \bigoplus_{{r} \geq 0}\HH^{r}(\nabla A)$ can be defined by}
\[
\begin{cases}
s_{1}
&\mapsto [h_{1}];\\
s_{2}
&\mapsto \lambda_{m} [h_{2}];\\
s_{3}
&\mapsto [h_{3}];
\end{cases}
\quad
\begin{cases}
s_{4}
&\mapsto [h_{4}];\\
t_{k} &\mapsto \left[\tau_{f_{k}}^{xyx}\right] \text{ for } 1 \leq k \leq n;\\
t_{m+1} &\mapsto \left[\tau_{g_{1}}^{yxy}\right].\\
\end{cases}
\]
Then $\phi'$ is surjective $k$-algebra morphism and 
$\Ker \phi' = \left( \left\langle s_{1}s_{4}-s_{2}s_{4}, s_{2}s_{3} \right\rangle + \bigoplus_{{p} \geq 3}(\Lambda(a,b))_{p} \right)$.
\textcolor{red}{This completes the proof.}
\end{enumerate}
\end{proof}

\section*{Acknowledgments}
\textcolor{red}{The authors are grateful to the referee for their useful comments that helped improve the paper. We would like to thank the past and present members of our laboratory, 
Shota Inoue, Masaki Matsuno and Yuto Kitamura, for helpful discussions.}
The first author was supported by JSPS Grant-in-Aid for Scientific Research (C) 24K06653. 


\begin{thebibliography}{99}

\bibitem{AS}
M. Artin and W.~F. Schelter, 
{\em Graded algebras of global dimension $3$}, Adv. in Math. {\bf 66} (1987), no.~2, 171--216. 

\bibitem{ASS}
I. Assem, D. Simson and A. Skowro\'nski, 
Elements of the representation theory of associative algebras. Vol. 1, 
London Mathematical Society Student Texts, 65, Cambridge Univ. Press, Cambridge, 2006. 

\bibitem{AZ}
M. Artin and J.~J. Zhang, 
{\em Noncommutative projective schemes, Adv. Math.} 
{\bf 109} (1994), no.~2, 228--287. 

\bibitem{Bel}
P. Belmans, 
{\em Hochschild cohomology of noncommutative planes and quadrics}, 
J. Noncommut. Geom. {\bf 13} (2019), no.~2, 769--795. 

\bibitem{BR}
G.~M. Benkart and T. Roby, 
{\em Down-up algebras}, J. Algebra {\bf 209} (1998), no.~1, 305--344. 

\bibitem{BP}
A.~I. Bondal and A. Polishchuk, 
{\em Homological properties of associative algebras: the method of helices},
Russian Acad. Sci. Izv. Math. {\bf 42} (1994), no.~2, 219--260; 
translated from Izv. Ross. Akad. Nauk Ser. Mat. {\bf 57} (1993), no.~2, 3--50. 

\bibitem{CM}
P.~A.~A.~B. Carvalho and I.~M. Musson, 
{\em Down-up algebras and their representation theory}, 
J. Algebra {\bf 228} (2000), no.~1, 286--310. 

\bibitem{GS}
E.~L. Green and N. Snashall, 
{\em Projective bimodule resolutions of an algebra and vanishing of the second Hochschild cohomology group}, Forum Math. {\bf 16} (2004), no.~1, 17--36. 

\bibitem{H}
D. Happel, 
{\em Hochschild cohomology of finite-dimensional algebras}, 
in {\it S\'eminaire d'Alg\`ebre Paul Dubreil et Marie-Paul Malliavin, 39\`eme Ann\'ee (Paris, 1987/1988)}, 108--126, Lecture Notes in Math., 1404, Springer, Berlin. 

\bibitem{H2}
D. Happel, 
{\em The trace of the Coxeter matrix and Hochschild cohomology}, 
Linear Algebra Appl. {\bf 258} (1997), 169--177. 

\bibitem{IU}
A. Itaba and K. Ueyama, 
{\em Hochschild cohomology related to graded down-up algebras with weights $(1,n)$}, 
J. Algebra Appl. {\bf 20} (2021), no.~7, Paper No. 2150131, 19 pp.

\bibitem{KMP}
E.~E. Kirkman, I.~M. Musson and D.~S. Passman, 
{\em Noetherian down-up algebras}, Proc. Amer. Math. Soc. {\bf 127} (1999), no.~11, 3161--3167.

\bibitem{KT}
O.~Y. Kushel and M. Tyaglov, 
{\em Circulants and critical points of polynomials}, 
J. Math. Anal. Appl. {\bf 439} (2016), no.~2, 634--650. 

\bibitem{LWZ}
O. Lezama, Y.~H. Wang and J.~J. Zhang, 
{\em Zariski cancellation problem for non-domain noncommutative algebras}, 
Math. Z. {\bf 292} (2019), no.~3-4, 1269--1290. 

\bibitem{M}
Y. Matsuoka, 
{\em Classroom Notes{\rm :} On a Proof of Hermite's Identity}, 
Amer. Math. Monthly {\bf 71} (1964), no.~10, 1115. 

\bibitem{MM}
H. Minamoto and I. Mori, 
{\em The structure of AS-Gorenstein algebras}, 
Adv. Math. {\bf 226} (2011), no.~5, 4061--4095. 

\bibitem{I}
A.~W. Ingleton, 
{\em The rank of circulant matrices}, 
J. London Math. Soc. {\bf 31} (1956), 632--635. 

\bibitem{dNV}
K. De~Naeghel and M. Van~den~Bergh, 
{\em Ideal classes of three-dimensional Sklyanin algebras}, 
J. Algebra {\bf 276} (2004), no.~2, 515--551. 

\bibitem{W}
S.~J. Witherspoon, Hochschild cohomology for algebras, Graduate Studies in Mathematics, 204, Amer. Math. Soc., Providence, RI, (2019).

\end{thebibliography}
\end{document}